\documentclass[11pt,leqno]{article}
\usepackage{amsthm,amsfonts,amssymb,amsmath,oldgerm}
\numberwithin{equation}{section}
\usepackage{fullpage}
\usepackage{graphicx}
\usepackage{subcaption}
\usepackage{color}

\newcommand\R{\mathbb R}
\newcommand\C{\mathbb C}
\def\eps{\varepsilon}
\newcommand{\CalE}{\mathcal{E}}
\newcommand{\CalF}{\mathcal{F}}
\newcommand{\CalH}{\mathcal{H}}
\newcommand{\CalL}{\mathcal{L}}
\newcommand{\CalO}{\mathcal{O}}
\newcommand{\diag}{{\rm diag }}
\newcommand{\RM}{\mathbb{R}}
\newcommand{\ZM}{\mathbb{Z}}
\newcommand{\CM}{\mathbb{C}}
\newcommand{\cn}{\operatorname{cn}}

\newcommand{\abs}[1]{\lvert#1\rvert}
\newcommand{\parens}[1]{\left( #1 \right)}
\newcommand{\brackets}[1]{\left[ #1 \right]}
\newcommand{\braces}[1]{\left\{ #1 \right\}}
\newcommand{\innerprod}[2]{\left\langle #1, #2 \right\rangle}
\newcommand{\norm}[1]{\left\lVert #1 \right\rVert}
\DeclareMathOperator{\range}{range}
\newcommand{\Z}{\mathbb{Z}}
\newcommand{\CalU}{\mathcal{U}}

\renewcommand{\Re}{\operatorname{Re}}
\renewcommand{\Im}{\operatorname{Im}}

\newtheorem{theorem}{Theorem}[section]
\newtheorem{proposition}[theorem]{Proposition}

\newtheorem{lemma}[theorem]{Lemma}
\theoremstyle{remark}
\newtheorem{remark}[theorem]{Remark}

\allowdisplaybreaks[3]

\title{Nondegeneracy and Stability of Antiperiodic Bound States for Fractional Nonlinear
Schr\"odinger Equations}

\author{Kyle~M.~Claassen\thanks{Department of Mathematics, Rose-Hulman Institute of Technology, 5500 Wabash Ave, Terre Haute, IN 47803; claassen@rose-hulman.edu} 
\quad \&\quad Mathew~A.~Johnson\thanks{Department of Mathematics, University of Kansas, 1460 Jayhawk Boulevard, 
Lawrence, KS 66045; matjohn@ku.edu} }

\date{\today}

\begin{document}

\maketitle

\begin{abstract}
We consider the existence and stability of real-valued, spatially antiperiodic standing wave solutions
to a family of nonlinear Schr\"odinger equations with fractional dispersion and power-law nonlinearity.  
As a key technical
result, we demonstrate that the associated linearized operator is nondegenerate when restricted
to antiperiodic perturbations, i.e. that its kernel is generated by the translational and 
gauge symmetries of the governing evolution equation.  In the process, we provide a characterization
of the antiperiodic ground state eigenfunctions for linear fractional Schr\"odinger operators on $\mathbb{R}$ 
with real-valued, periodic potentials as well as a Sturm-Liouville type oscillation theory for 
the higher antiperiodic eigenfunctions.
\end{abstract}

\section{Introduction}  \label{S:intro}

In this paper, we consider the existence and stability properties of real-valued, spatially periodic 
solutions to a class of fractional nonlinear Schr\"odinger equations (fNLS) of the form
\begin{equation}  \label{E:FNLS1}
i u_t - \Lambda^\alpha u + \gamma |u|^{2\sigma} u = 0,\quad (x,t)\in\mathbb{R}^2
\end{equation}
where subscripts denote partial differentiation.  Here and throughout, $u=u(x,t)$ is a 
generally complex-valued function, and the pseudodifferential operator $\Lambda:=\sqrt{-\partial_x^2}$,
referred to as Calderon's operator, is of first order and, acting on $2T$-periodic functions, is defined by its Fourier multiplier
via
\[
\widehat{\Lambda f}(n)=\frac{\pi|n|}{T}\hat{f}(n),\quad n\in\mathbb{Z}.
\]
Further, $\gamma=\pm 1$ distinguishes between focusing (attracting) $\gamma=+1$ and defocusing (repulsive)
$\gamma=-1$ nonlinearities.  

The parameter $\alpha\in(0,2]$ describes the fractional dispersive nature of the equation.  
When $\alpha=2$, the operator $\Lambda^2=-\partial_x^2$ denotes the local (positive) Laplacian. 
In this classical case, \eqref{E:FNLS1} reduces to the well-studied nonlinear Schr\"odinger equation (NLS),
which is known to serve as canonical model describing weakly nonlinear wave propagation
in dispersive media; see, for example, \cite{SS99}.  
When $\alpha\in(0,2)$, $\Lambda^\alpha$ denotes the so-called
\emph{fractional Laplacian}, which arises naturally in a variety of applications including 
the continuum limit of discrete models with long range interaction \cite{KLS}, 
dislocation dynamics in crystals \cite{CFM}, mathematical biology \cite{MCGJR99}, water
wave dynamics \cite{IP14}, and financial mathematics \cite{CT}; see also \cite{BV15} for a recent
discussion on applications.
For such $\alpha$, the \emph{nonlocal} fNLS \eqref{E:FNLS1} has been introduced by Laskin \cite{Las1} in the 
context of fractional quantum mechanics, in which one generalizes the standard Feynman path integral from Brownian-like
to L\'{e}vy-like quantum mechanical paths.  A rigorous derivation in the context of charge transport in 
bio polymers (like DNA), can be found in \cite{KLS}.  Finally, we point out the case $\alpha=1$
may be thought to describe the relativistic dispersion relation $\omega(\xi)=\sqrt{|\xi|^2+m^2}$,
an observation recently utilized in the mathematical description of Boson-stars; see \cite{FJS}.

Throughout our analysis, we will be concerned with solutions of the form
\begin{equation}\label{E:bound_state}
u(x,t) = e^{i\omega t}\phi(x-ct;c),
\end{equation}
where $\omega,c\in\mathbb{R}$ are parameters and $\phi$ is a bounded
solution to the (generally) nonlocal profile equation
\[
\Lambda^\alpha\phi+\omega \phi+ic\phi'-\gamma|\phi|^{2\sigma}\phi=0,
\]
where $'$ denotes differentiation with respect to the spatial variable.
When $c=0$, the focusing fNLS is well-known to admit standing solitary waves that are asymptotic to zero at spatial infinity.  Among such solitary wave solutions,
specific attention is often paid to the positive, radially symmetric
solutions typically referred to as ``ground states".  The stability of such ground states dates
back to the work of Cazenave and Lions \cite{CL82} and Weinstein \cite{Weinstein85,Wein86} on the classical
case $\alpha=2$, where the authors use the method of concentration compactness along with the 
construction of appropriate Lyapunov functionals.  For $\alpha\in(1,2]$, such ground states are known to be
orbitally stable provided the nonlinearity is energy sub-critical, i.e. if $0<\sigma<\alpha$; see
\cite{GH12,Wein86}, for example.
While no nontrivial localized solutions exist in the defocusing case $\gamma=-1$, it is known in the classical
case $\alpha=2$ to admit so-called black solitons of the form \eqref{E:bound_state} corresponding
to monotone front-like solutions asymptotic to constants as $x\to\pm\infty$.  The dynamics and stability of black
solitons has been studied in numerous works; see, for example, \cite{BGS08,GP15Blk,GS15}.  
For the fractional case, the authors plan to report on the existence, nondegeneracy,
and stability of black solitons in the defocusing case in a future work.

The stability of periodic standing waves of \eqref{E:FNLS1} is considerably less understood 
than their asymptotically constant counterparts, even in the classical $\alpha=2$ case.  
Recall that when $\alpha=2$ and $\sigma=1$, Rowlands \cite{Rowlands74} formally demonstrated the spectral instability
to long-wavelength (i.e. modulational) perturbations of all periodic standing waves in the focusing case, along
with the stability (to long-wavelength perturbations) of such waves in the defocusing case.  Rowlands' results
were later rigorously established by Gallay and Haragus \cite{GHsmall} for small amplitude waves, and
by Gustafson, Le Coz, and Tsai \cite{GCT} and Deconinck and Segal \cite{DS16} for arbitrary amplitude waves.
The spectral stability to arbitrary bounded perturbations 
of periodic standing waves in the defocusing NLS ($\alpha=2$) was later shown
by Gallay and Haragus \cite{GHsmall} for small amplitude waves and by Bottman, Deconinck, and Nivala \cite{BDN11},
again using complete integrability, for waves of arbitrary amplitude.
These observations motivate us to restrict much of our
attention to the defocusing case $\gamma=-1$, as we expect such waves in the focusing case  
to be modulationally unstable
for general $\alpha$, although, this has not been verified.
In fact, to the authors' knowledge there has been no rigorous study into the dynamics of such
waves in the fractional case.  

\

In this work, we will be concerned with the \emph{nonlinear stability} of periodic standing waves of \eqref{E:FNLS1}
in the genuinely nonlocal case $\alpha\in(0,2)$.  We mark that since the governing 
evolution equation is invariant under phase rotation and spatial translation, i.e. the map
\[
u(x,t)\mapsto e^{i\beta}u(x-x_0,t),\quad x_0,\beta\in\mathbb{R}
\]
preserves the class of solutions of \eqref{E:FNLS1}, we should only expect
stability up to these invariances.  Such \emph{orbital stability} results for solutions
of \eqref{E:FNLS1} have been obtained in the local case $\alpha=2$.
The first result in this direction we are aware
of in the periodic case was due to Angulo-Pava \cite{Pava07}, where, in the focusing case with $\alpha=2$, 
the orbital stability of dnoidal type (hence strictly positive) standing
waves was established to co-periodic perturbations, i.e. to perturbations with the same period as the underlying wave.  
Pava's analysis of dnoidal waves relied on a direct adaptation of the classical approach of orbital stability by Grillakis, Shatah, and Strauss \cite{GSS1,GSS2}.
In particular, in \cite[Theorem 3.1]{Pava07} it was shown that the Hessian of the associated Lagrangian at such a dnoidal wave, when acting on co-periodic
perturbations, has exactly one negative eigenvalue, a double (semi-simple) eigenvalue at zero generated by the above (continuous) invariances  of the PDE, and the rest of the spectrum
is positive and uniformly bounded from below away from zero.  In the case of cnoidal (hence sign changing) waves, however, Pava showed in \cite[Theorem 3.2, Theorem 3.4]{Pava07} that the Hessian operator
\emph{has three negative eigenvalues} when acting on co-periodic perturbations, thus invalidating the structurial structural hypotheses of \cite{GSS1,GSS2}.
This issue was later resolved in the $\alpha=2$ case by Gallay and Haragus \cite{GH07},
where the authors demonstrated the orbital stability of cnoidal waves in the defocusing cubic NLS ($\alpha=2$, $\sigma=1$)
to perturbations with the same period as the \emph{modulus} of the underlying wave.  In particular, in \cite{GH07} the authors verified that, when $\alpha=2$,
the Hessian of the Lagrangian about a $2T$-periodic cnoidal wave has only \emph{one} negative eigenvalue when acting on $T$-antiperiodic\footnote{$T$-antiperiodic functions are defined in \eqref{E:antiperiodic_definition} and the surrounding discussion.  Notice that all $T$-antiperiodic functions are necessarialy $2T$-periodic.}
perturbations, and then establishing nonlinear orbital stability of such waves to this more restrictive class of perturbations via the theory of Grillakis, Shatah, and Strauss.  
See also the more recent work \cite{GP15Cn} where the restriction to antiperiodic perturbations is removed through the use
of the completely integrable structure of the cubic, defocusing NLS.

Here, we take matters further and study the existence and nonlinear stability of periodic
standing waves of the fNLS \eqref{E:FNLS1} with fractional dispersion\footnote{The restriction to $\alpha>1$ is necessitated by our existence theory,
which requires the energy space $H^{\alpha/2}$ to be a Banach algebra.  It is not known if such waves exist when $\alpha<1$.  Similarly, the restriction $\alpha<2$ is necessary
to in the development of the antiperiodic ground state theory in Section \ref{S:GS_Theory}: see Remark 3.5 below.  Note the case $\alpha=2$ is not included in our analysis since it 
can be treated by ODE-based techniques as in \cite{Pava07,GH07}.  It is not known if our results remain true when $\alpha>2$.} $\alpha\in(1,2)$, provided, following
\cite{GH07}, we appropriately restrict the class of perturbations.  In the defocusing case, 
where we primarily restrict our attention,
we will show in Section \ref{S:existence} that, for each $T>0$, there exists a three-parameter family
of real-valued, $T$-\emph{antiperiodic} standing waves, i.e. $2T$-periodic
waves with
\begin{equation}
\phi(x+T)=-\phi(x),  \label{E:antiperiodic_definition}
\end{equation}
arising as local minima of the Hamiltonian energy subject to conservation of momentum: see 
Proposition \ref{P:min} and Lemma \ref{L:c=0_prop}.  
Once existence is established, our main goal is to establish the nonlinear stability
of such real-valued, $T$-antiperiodic standing waves to small $T$-antiperiodic perturbations: see
Theorem \ref{T:orbital_stability}.  
From the above discussion, this effectively extends the ``cnoidal-wave" analysis of Galley \& Haragus \cite{GH07} for the classical NLS  ($\alpha=2$)
to the fractional case $\alpha\in(1,2)$.

A key step in our stability analysis is to show that the Hessian of the Hamiltonian energy
is \emph{nondegenerate} at such an antiperiodic, local constrained minimizer of the defocusing fNLS;
that is, that the kernel is generated only by spatial translations and phase rotations.
The nondegeneracy of the linearization is known to play an important role in the stability
of traveling and standing waves (see \cite{Wein86}, \cite{Lin08}, and \cite{GH07})
and in the blowup analysis (see \cite{KMR11}, \cite{SS99}, for instance) of the related dynamical equation.
In the case of the classical NLS with cubic nonlinearity, 
the nondegeneracy at such antiperiodic standing wave solutions was established
by Gallay and Haragus \cite[Proposition 3.2]{GH07}.  
Their proof, however, fundamentally relies on ODE techniques, in particular
on the Sturm-Liouville theory for ODEs and a-priori bounds on the number of linearly independent solutions to the linearized equations, and is hence not directly applicable to the nonlocal case $\alpha\in(0,2)$.
Nevertheless, Frank and Lenzmann \cite{FL13} recently established
the nondegeneracy of \emph{solitary} waves for a family of nonlocal evolution equations, including
the focusing fNLS.  Their analysis relied on the development of a suitable substitute for the Sturm-Liouville
theory, following from the characterization of the fractional Laplacian as a Dirichlet-to-Neumann operator
for a \emph{local} elliptic problem in the upper half-plane, allowing them to bound from above the number of sign changes of eigenfunctions for fractional linear Schr\"odinger
operators on the line.  This oscillation theory was recently extended to the periodic
setting in \cite{HJ15}, where the authors considered the orbital stability of periodic traveling 
waves of the fractional gKdV equation.  

While the oscillation theory in \cite{HJ15} seems to apply
directly to periodic standing waves in the focusing fNLS, see Remark \ref{R:focusing_signdefinite}, 
it requires considerable modification in the 
defocusing case, accounting for the $T$-antiperiodicity of the eigenfunctions compared with the 
$T$-\emph{periodicity} of the potential in the associated linear Schr\"odinger operators.  We 
point out that even in the classical $\alpha=2$ case, antiperiodic ground states for linear Schr\"odinger
operators \emph{need not be simple}.  Indeed, it is not difficult to cook up examples of potentials for which the 
associated Schr\"odinger operator will have an antiperiodic ground state with 
multiplicity two; see \cite{MW79} for instance.  We handle this difficulty
by restricting such operators to even and odd antiperiodic subspaces, demonstrating that the associated
linear semigroup is positivity improving on these individual subspaces.  A twist on 
standard Perron-Frobenius arguments
then yields a characterization of the antiperiodic ground states of such fractional linear Schr\"odinger operators 
restricted to even and odd functions: see Theorem \ref{T:ground}.  
Further, using antiperiodic rearrangement inequalities developed in
Appendix \ref{A:Polya}, we demonstrate that the ordering of the even and odd antiperiodic ground states
for a fractional linear Schr\"odinger operator depends explicitly on the monotonicity properties of the
real-valued periodic potential: see Proposition \ref{P:gsordering}.
Once the appropriate ground state theories are developed, we
then provide a suitable oscillation theory for higher antiperiodic eigenfunctions by following the arguments
in \cite{FL13,HJ15}: see Lemma \ref{L:oscillation}.

We emphasize that the realness of the $T$-antiperiodic solutions
$\phi$ discussed above is \emph{absolutely crucial} to our analysis, 
guaranteeing that the Hessian of the Lagrangian functional
encoding solutions of the profile equation as critical points acts as a diagonal operator
on $L^2_{\rm a}(0,T)$.  This diagonal property reduces the nondegeneracy
analysis to the study of two scalar linear fractional Schr\"odinger operators, which is precisely
the setting where our techniques from Section \ref{S:nondegeneracy} apply.
In the ``nontrivial phase" case, corresponding to genuinely complex-valued solutions $\phi$, 
the Hessian operator couples the real and imaginary parts of the perturbations,
invalidating the strategy and techniques contained in this paper: see equation \eqref{E:Hessian} below.
The corresponding nondegeneracy and stability analysis
for the nontrivial phase solutions is completely open in the nonlocal case 
and is a \emph{very} interesting direction for future research.  To our knowledge,
the only nondegeneracy result known in the nontrivial phase case was provided by in \cite{GH07} using,
as discussed before, restrictive ODE techniques.

In the focusing case, we construct real-valued antiperiodic solutions through a different 
constrained minimization procedure, producing solutions that may or may not be constrained energy minimizers:
see Proposition \ref{P:focusing_existence}.
By applying our techniques developed for the defocusing case, we establish in Proposition \ref{P:focusing_nondegeneracy}
the nondegeneracy of all of these waves \emph{independent} of whether they are constrained energy minimizers.  
As an application, we then characterize the stability of
these waves in terms of variations of the $L^2$ norm with respect to an appropriate parameter: see
Theorem \ref{T:focusing_orbital_stability} and Proposition \ref{P:focusing_instability}.

\

The outline of the paper is as follows.  In Section \ref{S:existence} we use variational arguments
to establish the existence of antiperiodic solutions of the profile equation \eqref{E:profile_equation} in the 
defocusing case.  Our main existence result is Proposition \ref{P:min}.  Section \ref{S:nondegeneracy} is devoted to the proof of our main technical result
Proposition \ref{P:nondegeneracy}, which establishes the nondegeneracy of the real-valued, antiperiodic solutions constructed in Section \ref{S:existence}.
As discussed above, the proof of Proposition \ref{P:nondegeneracy} hinges on the development of an antiperiodic ground state theory for fractional
Schr\"odinger operators (see Theorem \ref{T:ground} and Proposition \ref{P:gsordering}), as well as the development of an antiperiodic oscillation, i.e. Sturm-Liouville type, 
theory (see  Lemma \ref{L:oscillation}).  With the nondegeneracy result Proposition \ref{P:nondegeneracy} established, we then prove the nonlinear orbital stability
of these real-valued, antiperiodic waves in Section \ref{S:stability_of_minimizers}: see Theorem \ref{T:orbital_stability} for a precise statement of this stability result. 
Finally, in Section \ref{S:focusing_case} we discuss an extension of our work to the focusing case.
Appendix \ref{A:Polya} contains proofs of the relevant rearrangement inequalities used in the development of the antiperiodic ground state theory used in Section \ref{S:nondegeneracy}.

\section{Existence of Constrained Local Minimizers in Defocusing Case}  \label{S:existence}

We begin our analysis by establishing the existence of periodic waves of the form
\eqref{E:bound_state} of the \emph{defocusing} ($\gamma=-1$) fNLS \eqref{E:FNLS1}.
Substituting the traveling wave ansatz \eqref{E:bound_state} 
into \eqref{E:FNLS1} yields the nonlocal \textit{profile equation}
\begin{equation}  \label{E:profile_equation}
\Lambda^\alpha \phi +\omega\phi + ic\phi'+ |\phi|^{2\sigma}\phi = 0,~~\omega,c\in\mathbb{R}.
\end{equation}
where here $\phi$ is generally a complex-valued function and $\sigma>0$; further restrictions will
be placed on $c$, $\omega$, and $\sigma$ later.
Here and throughout, given a finite period $T>0$
we consider for each $\alpha>0$ the operator $\Lambda^\alpha$ as 
a closed operator on 
\[
L^2_{\rm per}([0,2T];\CM):=\left\{f\in L^2_{\rm loc}(\mathbb{R};\CM):f(x+2T)=f(x)~~\forall x\in\mathbb{R}\right\}
\]
with dense domain $H^\alpha_{\rm per}([0,2T];\CM)$, defined
via its Fourier series as
\[
\Lambda^\alpha f(x) = \sum_{n\in\mathbb{Z}\setminus\{0\}}\left| \frac{\pi n}{T} \right|^\alpha e^{\pi inx/T}\hat{f}(n),~~\alpha\geq 0.
\]
We will be interested primarily in \emph{real-valued, standing} wave solutions of \eqref{E:FNLS1}, in which
case $c=0$.  To motivate the expected structure of such solutions, 
we note that when $\alpha=2$ and $c=0$
the profile equation \eqref{E:profile_equation} is integrable and, upon integration, can be expressed as
\[
\frac{1}{2}\left(\phi'\right)^2=H-V(\phi;\omega)
\]
where $H\in\mathbb{R}$ denotes the ODE energy and 
\[
V(\phi;\omega):=-\frac{\omega}{2}\phi^2-\frac{1}{2\sigma+2}\phi^{2\sigma+2}
\]
denotes the effective potential energy.
Observe that the potential $V$ is even for every $\omega\in\RM$ and,
possesses a unique local minimum when $\omega<0$,
yielding the existence of a one-parameter family, parameterized by $H$, of periodic orbits\footnote{When $\omega>0$,
the potential $V$ is strictly decreasing on $(0,\infty)$ and hence no nontrivial bounded solutions
exist in this case.}
oscillating symmetrically about the equilibrium solution $\phi=0$; see Figure \ref{F:defocusing_potential}.    
Further, up to translations these waves can be chosen to be even and \emph{antiperiodic}, i.e.
they satisfy
\[
\phi(x+T)=-\phi(x)
\]
where $2T>0$ denotes the fundamental period of $\phi$.
Note that while such solutions can be expressed explicitly in terms 
of the Jacobi elliptic function $\cn$, the authors
are unaware of such an explicit solution formula for $\alpha\in(0,2)$.
Nevertheless, for each $\alpha\in(1,2)$ and $T>0$ we expect to be able to construct a three-parameter family of 
real-valued, $T$-antiperiodic solutions of the profile equation \eqref{E:profile_equation}.

\begin{figure}[t]
\begin{center}
\includegraphics[scale=0.12]{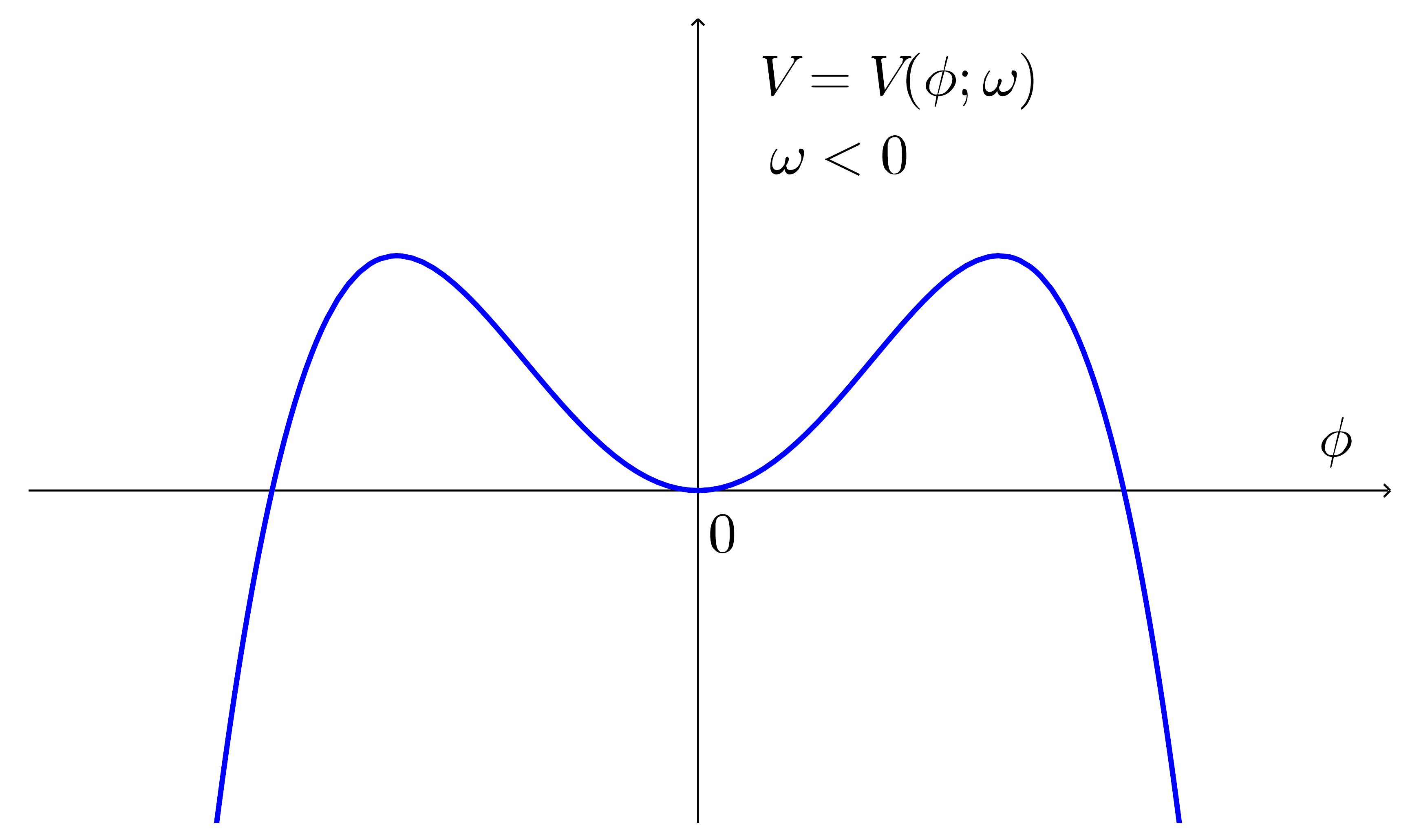}
\end{center}
\caption{The effective potential $V(\phi;\omega)$ for the defocusing NLS with $\omega < 0$.}
\label{F:defocusing_potential}
\end{figure}

In the absence of the integrable ODE structure present in the case $\alpha=2$, we will construct real-valued,
antiperiodic solutions of \eqref{E:profile_equation} with $c=0$ through a constrained minimization argument. 
We emphasize, however, that in our forthcoming stability analysis it will be important that our real-valued
antiperiodic solutions be (locally) embedded into a larger family of \emph{complex-valued, traveling} waves
with $c\neq 0$; see Remark \ref{r:travel} below.  
In the local case $\alpha=2$, this embedding is easily handled due to the 
existence of an exact Galilean invariance: precisely, if $u(x,t)$ is a solution of \eqref{E:FNLS1} with $\alpha=2$,
then so is
\begin{equation}\label{E:galilean}
\mathcal{G}_cu(x,t)= e^{i{\textstyle\frac{c}{2}}x-i\textstyle{\frac{c^2}{4}}t} u(x-ct,t)
\end{equation}
for each wave speed $c\neq 0$.  However, when $\alpha\neq 2$ no such exact Galilean invariance exists.
As a result, our variational arguments below will actually construct a general class of antiperiodic
traveling wave solutions of \eqref{E:FNLS1} of the form \eqref{E:bound_state} with $|c|$ sufficiently small,
and then show through various rearrangement arguments that when $c=0$ the resulting waves
can be chosen to be real-valued with appropriate dependencies on the wave speed.

To this end, given a $T>0$ we work throughout in the $L^2$-based Lebesgue and Sobolev spaces over the 
antiperiodic interval $[0,T]$.  Define the \emph{real} vector space
\[
L^2_{\rm a}([0,T];\mathbb{C}):=\left\{f\in L^2_{\rm loc}(\RM;\CM):f(x+T)=-f(x)~~\forall x\in\RM\right\}
\]
equipped with inner product $\left<u,v\right>=\Re\int_0^T u \bar{v}~dx$, and 
for each $\alpha\in(0,2)$ define
\begin{equation}\label{D:Y}
H^{\alpha/2}_{\rm a}([0,T];\mathbb{C}) := \left\{ f \in H^{\alpha/2}_{\rm loc}(\R;\mathbb{C}) : 
		f\in L^2_{\rm a}([0,T];\CM)\right\}
\end{equation}
considered as a \emph{real} vector space with inner product
\[
(u,v) := \Re \int_0^{T} (u \bar{v} + \Lambda^{\alpha/2}u \, \overline{\Lambda^{\alpha/2}v}) \, dx.
\]
By Sobolev embedding, it is trivial to see that $H^{\alpha/2}_{\rm a}(0,T)$ is a closed subspace
of $H^{\alpha/2}_{\rm per}(0,T)$, and hence is itself a Hilbert space, for all $\alpha>1$.
Furthermore, we identify the dual space $H^{\alpha/2}_{\rm a}(0,T)^*$ with 
$H^{-\alpha/2}_{\rm a}(0,T)$ via the pairing
\begin{equation}\label{D:dual}
\innerprod{u}{v} := \Re \int_0^{T} u \bar{v} \, dx,\quad u\in H^{\alpha/2}_{\rm a}(0,T)^*,~~
v\in H^{\alpha/2}_{\rm a}(0,T).
\end{equation}
Throughout, unless otherwise stated, we will use the slight abuse of notation that
\[
H^s_{*}(0,T):=H^s_*([0,T];\mathbb{C})
\]
where here $*$ stands for any of ``per", or ``a".  
We note that we may at times work with the subspace of \emph{real-valued} functions
$H^{\alpha/2}_{\rm a}([0,T];\mathbb{R})$ in $H^{\alpha/2}_{\rm per}(0,T)$;
when the choice of scalar field is irrelevant or obvious from context, we will simply
write $H^{\alpha/2}_{\rm a}(0,T)$ for these spaces.

To begin with our existence theory, we consider $\alpha>1$ and define the functionals
\[
K(u):=\frac{1}{2}\int_0^T|\Lambda^{\alpha/2}u|^2dx,\quad 
P(u):=\frac{1}{2\sigma+2}\int_0^T|u|^{2\sigma+2}dx
\]
on $H^{\alpha/2}_{\rm a}(0,T)$, which we refer to as the \emph{kinetic} and \emph{potential}
energies, respectively.  For $\alpha>1$ the fNLS \eqref{E:FNLS1} admits 
the conserved quantities
\begin{align*}
\CalH(u) &:= K(u)+P(u)=\frac{1}{2} \int_0^{T} \parens{ \abs{\Lambda^{\alpha/2} u}^2 
			+\frac{1}{\sigma+1} \abs{u}^{2\sigma+2} } dx,\\
Q(u)&:=\frac{1}{2}\int_0^T|u|^2~dx,\qquad N(u):=\frac{i}{2}\int_0^T \overline{\Lambda^{1/2} u}~H\Lambda^{1/2} u~dx
\end{align*}
which we refer to as the \emph{Hamiltonian (energy)}, \emph{charge}, and 
\emph{(angular) momentum}\footnote{On smooth solutions, the momentum
functional is defined as $N(u)=\frac{i}{2}\int_0^T \bar{u}(z)u'(z)dz$.  
Using the identity $\partial_x=H\Lambda$, we may consider $u_x$ to be well-defined
on $H^{\alpha/2}_{\rm a}(0,T)$ in the sense of distributions.}, respectively.  Conservation
of $\mathcal{H}$ comes from the fact that \eqref{E:FNLS1} is autonomous in time, while conservation
of $Q$ and $N$ is due to the phase and translational invariance of \eqref{E:FNLS1}, respectively.
Above, $H$ denotes the Hilbert transform, a bounded linear map from $L^2_{\rm a}(0,T)\to L^2_{\rm a}(0,T)$
with unit norm.

For a general $\alpha\in(1,2]$, it is clear that $\mathcal{H}$, $Q$, and $N$
are smooth functionals on $H^{\alpha/2}_{\rm a}(0,T)$.  Further, their first order variational
derivatives are smooth maps from $H^{\alpha/2}_{\rm a}(0,T)$ into $H^{\alpha/2}_{\rm a}(0,T)^*$
given explicitly by
\[
\delta\mathcal{H}(u)=\Lambda^{\alpha}u+ |u|^{2\sigma}u,\quad\delta Q(u)=u,\quad\delta N(u)=iu_x.
\]
It follows from \eqref{E:profile_equation} that $T$-antiperiodic standing waves of \eqref{E:FNLS1} 
arise as critical points of the Lagrangian functional
\[
H^{\alpha/2}_{\rm a}(0,T)\ni u\mapsto \mathcal{H}(u)+\omega Q(u)+cN(u)
\]
for some $\omega,c\in\mathbb{R}$.
It is natural now to treat the parameters $\omega$ and  $c$ as Lagrange multipliers, and search
for solutions of \eqref{E:profile_equation} as critical points of $\mathcal{H}$ subject
to the conservation of $Q$ and $N$.  However, below we need precise information on the range
of values of $c$ for which such a critical point exists.  Consequently we find it more appropriate
to treat the wave speed as a free-parameter and to attempt to construct solutions of \eqref{E:profile_equation}
for a \emph{fixed} $c$ as critical points of the functional
\[
F_c(u):=\mathcal{H}(u)+cN(u)
\]
subject to fixed $Q$.  The fact that such critical points exist as constrained minimizers is guaranteed
by the following.

\begin{proposition}\label{P:min}
Let $\alpha\in(1,2)$ and $T,\sigma>0$ be fixed in the defocusing ($\gamma=-1$) fNLS \eqref{E:FNLS1}.  
For each $\mu>0$ define the constraint space
\[
\mathcal{A}_\mu:=\left\{u\in H^{\alpha/2}_{\rm a}(0,T):Q(u)=\mu\right\}.
\]
Then for each $\mu>0$ and $|c| < c_* := \left(\frac{\pi}{T}\right)^{\alpha-1}$ there exists a nontrivial $\phi=\phi(\cdot;c,\mu)\in \mathcal{A}_\mu$ such
that 
\[
F_c(\phi)=\min_{u\in\mathcal{A}_\mu}F_c(u),
\]
and $\phi(\cdot;c,\mu)$ satisfies \eqref{E:profile_equation} for some $\omega=\omega(c,\mu)\in\RM$ 
in the sense of distributions. 
Moreover, the function $\phi$ belongs to $H^\infty_{\rm a}(0,T)$ and minimizes the Lagrangian
functional
\begin{equation}\label{E:lag}
\mathcal{E}(u;c,\mu):=\mathcal{H}(u)+\omega(c,\mu)Q(u)+cN(u)
\end{equation}
over $H^{\alpha/2}_{\rm a}(0,T)$ subject to the fixed $Q$ and $N$; specifically
\[
\mathcal{E}(\phi;c,\mu)=
\inf\left\{\mathcal{E}(\psi;c,\mu):\psi\in \mathcal{A}_\mu,~~N(\psi)=N(\phi)\right\}.
\]
\end{proposition}

\begin{proof}
First, for $\alpha>1$ we observe Cauchy-Schwarz and the continuity of the embedding 
$H^{\alpha/2}_{\rm a}(0,T)\subset H^{1/2}_{\rm a}(0,T)$  imply that
$|N(u)|\leq \left( \frac{T}{\pi} \right)^{\alpha-1} K(u)$ for all $u\in H^{\alpha/2}_{\rm a}(0,T)$.  It follows that 
the functional $F_c$ is bounded below on $\mathcal{A}$ for all $|c|\leq \left( \frac{\pi}{T} \right)^{\alpha-1}$:
\[
F_c(u)=K(u)+P(u)+cN(u)\geq \left( 1 - |c| \left( \frac{T}{\pi} \right)^{\alpha-1} \right)K(u).
\]
Thus, if $|c| \leq \left( \frac{\pi}{T} \right)^{\alpha-1}$ then the quantity $\lambda:=\inf_{u\in\mathcal{A}_\mu}F_c(u)$
is well defined and finite, hence there exists a minimizing sequence $\{u_k\}_{k=1}^\infty\subset\mathcal{A}_\mu$
such that $\lim_{k\to\infty}F_c(u_k)=\lambda$.  Furthermore, using that
$\|u\|_{H^{\alpha/2}(0,T)}^2=2K(u)+2Q(u)$ we have
\[
\frac{1}{2}\left( 1 - \left( \frac{T}{\pi} \right)^{\alpha-1} |c| \right)\|u_k\|_{H^{\alpha/2}(0,T)}^2\leq F_c(u_k) + \left(1 - \left( \frac{T}{\pi} \right)^{\alpha-1} |c| \right)\mu,
\]
from which it follows that if $|c| < \left( \frac{\pi}{T} \right)^{\alpha-1}$ then $\{u_k\}_{k=1}^\infty$ is a bounded sequence
in $H^{\alpha/2}_{\rm a}(0,T)$.  By Banach-Alaoglu and the fact that weak limits are unique, we may thus 
extract a subsequence $\{u_{k_j}\}_{j=1}^\infty$ and a function 
$\phi\in H^{\alpha/2}_{\rm a}(0,T)$ such that $u_{k_j}$ converges weakly to $\phi$ in $H^{\alpha/2}_{\rm a}(0,T)$
and strongly (in norm) to $\phi$ in $L^2_{\rm a}(0,T)$.  Since $H^{\alpha/2}_{\rm a}(0,T)$
is compactly embedded in $L^{2(\sigma+1)}_{\rm a}(0,T)$ for any $\sigma>0$ by Sobolev embedding, we have that
$P$ is a compact operator on $H^{\alpha/2}_{\rm a}(0,T)$.  Together with the fact that $K$ is 
lower semicontinuous with respect to weak convergence in $H^{\alpha/2}_{\rm a}(0,T)$ , the above observations
imply that
\begin{equation}\label{E:prelimest}
\liminf_{j\to\infty}\mathcal{H}(u_{k_j})=\liminf_{j\to\infty}\left(K(u_{k_j})+P(u_{k_j})\right)\geq \mathcal{H}(\phi),
	\quad Q(\phi)=\mu.
\end{equation}
Finally, noting that $H^{\alpha/2}_{\rm a}(0,T)$ is compactly
embedded into $H^{1/2}_{\rm a}(0,T)$ by Rellich-Kondrachov, the sequence $\{u_{k_j}\}$ furthermore
converges to $\phi$ strongly in $H^{1/2}_{\rm a}(0,T)$.  From the estimate
\[
\left|N(u)-N(v)\right| \leq\frac{1}{2}\left(\|u\|_{H^{1/2}(0,T)}+\|v\|_{H^{1/2}(0,T)}\right)\|u-v\|_{H^{1/2}(0,T)},
\]
it follows that $N(u_{k_j})\to N(\phi)$ as $j\to\infty$.  It now follows from \eqref{E:prelimest} 
that $\phi\in\mathcal{A}_\mu$ and
\[
\lambda \leq F_c(\phi)\leq\liminf_{j\to\infty} F_c(u_{k_j})=\lambda,
\]
so that $F_c(\phi)=\lambda$.  Clearly then $\phi$ solves \eqref{E:profile_equation} for some
$\omega\in\RM$.

To see that $\phi(\cdot;c,\mu)$ minimizes $\mathcal{E}(\cdot;\mu ,c)$ over $H^{\alpha/2}_{\rm a}(0,T)$ constrained
to $Q(u)=\mu$ and $N(u)=N(\phi)$, simply observe that $\mathcal{E}=F_c+\omega(c,\mu) Q$ so
that $\phi$ clearly minimizes $\mathcal{E}(\cdot;c,\mu)$ over $\mathcal{A}_\mu$.
In particular, $\mathcal{E}(\phi;c,\mu)\leq\mathcal{E}(\psi;c,\mu)$ for all
$\psi\in\mathcal{A}_\mu$ with $N(\psi)=N(\phi)$.

It remains to establish the smoothness of the solution $\phi$.  By construction, we know
that $\phi\in H^{\alpha/2}_{\rm a}(0,T)$.  To show that $\phi\in H^{\alpha}_{\rm a}(0,T)$, 
notice that for any $|c| < \left( \frac{\pi}{T} \right)^{\alpha-1}$ the profile equation can be written as
\[
-\phi=(\Lambda^{\alpha}+ic\partial_x)^{-1}\left(\omega\phi+|\phi|^{2\sigma}\phi\right),
\]
where the operator $(\Lambda^\alpha+ic\partial_x)^{-1}$ is well-defined since $\phi$ has zero
mean by antiperiodicity.  Since $H^{\alpha/2}_{\rm a}(0,T)\subset L^\infty(0,T)$ by Sobolev 
embedding, the Plancherel theorem yields
\begin{align*}
\left\|\Lambda^\alpha\phi\right\|_{L^2(0,T)}&=\left\|\frac{|n|^{\alpha}}{|n|^{\alpha} - \left( \frac{T}{\pi} \right)^{\alpha-1} cn}\left(\omega\hat\phi(n)
			+\widehat{|\phi|^{2\sigma}\phi}(n)\right)\right\|_{\ell^2(\ZM\setminus\{0\})} \\
&\leq C \left( \|\phi\|_{L^2(0,T)}+\| |\phi|^{2\sigma}\phi \|_{L^2(0,T)} \right) \\
&\leq C \left( \|\phi\|_{L^2(0,T)}+\|\phi\|_{L^\infty(0,T)}^{2\sigma}\|\phi\|_{L^2(0,T)} \right) <\infty,
\end{align*}
for some $C>0$ independent of $\phi$, and hence $\phi\in H^\alpha_{\rm a}(0,T)$.
Similarly, by the fractional chain rule \cite[(3.3)]{HS15} we have
\begin{align*}
\left\|\Lambda^{2\alpha}\phi\right\|_{L^2(0,T)}&\leq C \left( \|\Lambda^\alpha\phi\|_{L^2(0,T)}
		+\left\|\Lambda^\alpha (|\phi|^{2\sigma}\phi) \right\|_{L^2(0,T)} \right) \\
&\leq C \left(	\|\Lambda^\alpha\phi\|_{L^2(0,T)} + \|\phi\|_{L^\infty}^{2\sigma}\|\Lambda^\alpha\phi\|_{L^2(0,T)} \right) < \infty
\end{align*}
for some $C>0$ independent of $\phi$, so that $\phi\in H^{2\alpha}_{\rm a}(0,T)$.  Iterating,
we find that $\phi\in H^\infty_{\rm a}(0,T)$ as claimed.
\end{proof}

To recapitulate, for each $\alpha\in(1,2)$, $\sigma,T>0$, $|c| < c_*$ and $\mu>0$, Proposition \ref{P:min} produces a generally complex-valued function $\phi(\cdot;c,\mu)\in H^{\infty}_{\rm a}(0,T)$ and $\omega(c,\mu)\in\RM$ such 
that
\[
\Lambda^\alpha\phi+\omega(c,\mu)\phi+ic\phi'+|\phi|^{2\sigma}\phi=0,\quad Q(\phi)=\mu.
\]
In particular, incorporating phase and translation invariances, for each half-period $T>0$ we have 
constructed a four-parameter family of generally complex-valued $T$-antiperiodic smooth solutions of the
defocusing fNLS \eqref{E:FNLS1}:
\[
u(x,t;c,\mu,\theta,\zeta)= e^{i(\omega(c,\mu)t-\theta)}\phi(x-ct+\zeta;c,\mu)
\]
where $|c| < c_*$, $\mu>0$, $\theta\in[0,2\pi]$ and $\zeta\in\RM.$
These solutions are parameterized by the wave speed $c$, the charge $Q(u)$ of the wave, and the
parameters $\theta$ and $\zeta$ associated to the continuous Lie point symmetries 
of the governing nonlocal PDE.
As stated at the beginning of this section, our focus in the remainder of the paper is on the
\emph{standing} wave solutions of \eqref{E:FNLS1}, corresponding to the above solutions with $c=0$.
Properties of the functions $\phi$ and $\omega$ at $c=0$ are recorded in the following lemma.

\begin{lemma}\label{L:c=0_prop}
The function $\omega:(-c_*,c_*)\times\RM_+\to\RM$ constructed in 
Proposition \ref{P:min} is even in $c$ and, for each
$\mu>0$, the profile $\phi(\cdot;0,\mu)\in H^\infty_{\rm a}(0,T)$ 
can be taken to be real-valued, even, and strictly decreasing on $(0,T)$.  The profiles $\phi(\cdot;0,\mu)$  additionally satisfy the constrained
variational problem
\[
\mathcal{E}(\phi;0,\mu)=\inf\left\{\mathcal{E}(\psi;0,\mu):\psi\in H^{\alpha/2}_{\rm a}(0,T),~~
Q(\psi)=Q(\phi),~~N(\psi)=0\right\}.
\]
Furthermore, for each $\mu>0$ such that $\phi$ is differentiable in $c$ at $c=0$, the function $\frac{\partial\phi}{\partial c}(\cdot; 0,\mu)$ is purely imaginary.
\end{lemma}

\begin{remark}
In the classical case $\alpha=2$, the periodic waves may be shown through ODE (quadrature) techniques to
be continuously differentiable in both $\mu$ and $c$.  Furthermore, when $\alpha=2$
the evenness of $\omega(c,\mu)$ in $c$ and the realness of the
function $i\frac{\partial\phi}{\partial c}$ follow trivially from the Galilean 
invariance \eqref{E:galilean}.
\end{remark}

\begin{proof}
By construction, we know that 
\begin{equation}\label{p1}
\Lambda^\alpha\phi+\omega(c,\mu)\phi+ic\phi'+|\phi|^{2\sigma}\phi=0.
\end{equation}
Since $F_c(g)=F_{-c}(\bar{g})$ for all $g\in H^{\alpha/2}_a(0,T)$, it follows that $\bar\phi$ minimizes $F_{-c}$ with the same fixed $Q$.
Consequently, we have that
\begin{equation}\label{p2}
\Lambda^\alpha\overline\phi+\omega(-c,\mu)\overline\phi-ic\overline\phi'+|\overline\phi|^{2\sigma}\overline\phi=0.
\end{equation}
Multiplying \eqref{p1} by $\overline\phi$ and \eqref{p2} by $\phi$, integrating, and subtracting the results
we find
\[
\left(\omega(c,\mu)-\omega(-c,\mu)\right)\int_0^T |\phi|^2dx=0
\]
from which it follows that $\omega$ is an even function of $c$, as claimed.   Similarly, we find
\[
\Re \left( \phi(\cdot; c,\mu) \right) = \frac{1}{2}\left(\phi(\cdot; c,\mu)+\phi(\cdot; -c,\mu)\right),
\]
which is an even function of $c$.  Differentiating this relation with respect to $c$ at $c=0$
yields, assuming the derivative exists, $\Re\left(\frac{\partial\phi}{\partial c}(\cdot; 0,\mu)\right)=0$.  Similarly, we find that
$\Im \left( \phi(\cdot; c,\mu) \right)$ is an odd function of $c$ and hence 
$\phi(\cdot;0,\mu)$ is necessarily real-valued.
The parity and monotonicity of $\phi(\cdot;0,\mu)$ now follow from the $2T$-periodic rearrangement arguments outlined
in Appendix \ref{A:Polya}, while the fact that $\phi(\cdot;0,\mu)$ satisfies the stated
constrained minimization problem follows trivially from Proposition \ref{P:min} by observing that $N(f)=0$
for all real-valued $f\in H^{\alpha/2}_{\rm a}(0,T)$.  
\end{proof}

In conclusion, for fixed $\alpha\in(1,2)$ and $\sigma>0$, we have constructed for each $T>0$
a three-parameter family of \emph{real-valued, $T$-antiperiodic, even} solutions
of the profile equation \eqref{E:profile_equation} with $c=0$.  These profiles lead
to a three-parameter family of \emph{standing wave} solutions of \eqref{E:FNLS1} of the form
\[
u(x,t;\mu,\theta,\zeta) = e^{i(\omega(0,\mu)t+\theta)}\phi(x+\zeta;0,\mu)
\]
In the forthcoming analysis, we will restrict our attention to these real-valued profiles
with $c=0$.  However, as stated previously, the fact that such solutions belong to a larger class of complex-valued
traveling waves will be used heavily in the forthcoming analysis; see Remark \ref{r:travel}.  For notational simplicity,
in the sequel we will suppress the dependence of $\omega$ and $\phi$ on the wave speed
$c$ whenever it is clear from context that $c=0$.

\section{Nondegeneracy of the Linearization in the Defocusing Case}  \label{S:nondegeneracy}

Throughout this section, for each $\mu>0$ we let $\phi(x;\mu)$ denote
a real-valued, even, $T$-antiperiodic standing wave solution of 
the nonlocal profile equation \eqref{E:profile_equation}
with $c=0$ satisfying $Q(\phi)=\mu$, whose existence is guaranteed by Proposition \ref{P:min} and Lemma \ref{L:c=0_prop}, so that the function $u(x,t;\mu)=e^{i\omega(\mu)t}\phi(x;\mu)$ is a $T$-antiperiodic standing
wave solution of the defocusing ($\gamma=-1$) fNLS \eqref{E:FNLS1}
Moving to a co-rotating coordinate frame, the profile $\phi(\cdot;\mu)$ is thus 
a real-valued, $T$-antiperiodic equilibrium solution of the PDE
\begin{equation}\label{E:FNLSrot}
iu_t-\omega(\mu) u-\Lambda^\alpha u+\gamma|u|^{2\sigma}u=0,
\end{equation}
which can be rewritten as the Hamiltonian system
\[
u_t=-i \, \delta \mathcal{E}(u;0,\mu)
\]
acting on $L^2_{\rm per}(0,2T)$, where here $\mathcal{E}$ is the modified energy functional defined in \eqref{E:lag}.
For such Hamiltonian systems, it is well known that the local 
dynamics of \eqref{E:FNLSrot} near $\phi$, in particular its orbital stability
or instability, is intimately related to spectral properties of the second variation
of the energy functional
\begin{equation}\label{E:Hessian}
\delta^2 \mathcal{E}(\phi;c,\mu)=\Lambda^{\alpha} +\omega(\mu)+ic\partial_x- \gamma\abs{\phi}^{2\sigma} 
	- 2\gamma\sigma \abs{\phi}^{2\sigma-2}\phi\Re(\bar{\phi}\cdot)
\end{equation}
acting on appropriate subspaces of $L^2_{\rm per}(0,2T)$.  
Of particular importance, observe that the $T$-antiperiodicity of $\phi$ implies that the 
operator $\delta^2 \mathcal{E}(\phi)$ has $T$-periodic coefficients.  As we will see below,
however, the continuous Lie point symmetries of \eqref{E:FNLSrot} generate elements
of the kernel of $\delta^2 \mathcal{E}(\phi)$ that are $T$-antiperiodic, and hence zero 
is an isolated eigenvalue of $\delta^2 \mathcal{E}(\phi)$ acting on $L^2_{\rm a}(0,T)$
with finite multiplicity.  In the forthcoming analysis, we restrict our attention
to $T$-antiperiodic 
perturbations\footnote{This restriction is indeed quite strong.  However, we will see that it is necessary in order to employ
the techniques of Grillakis-Shatah and Strauss \cite{GSS1,GSS2} to conclude orbital stability.}
of the underling wave $\phi$, thus requiring 
a detailed spectral analysis of the operator $\delta^2 \mathcal{E}(\phi)$ acting on $L^2_{\rm a}(0,T)$.

To aid in our analysis, we find it convenient to decompose the action of $\delta^2 \mathcal{E}(\phi)$ into 
real and imaginary parts.  Restricting our attention to the real-valued, stationary ($c=0$) solutions
$\phi$ constructed in Lemma \ref{L:c=0_prop}, the operator $\delta^2\mathcal{E}(\phi)$ acts as a diagonal
operator on $L^2_{\rm a}(0,T)$.  Indeed, for a given $v\in H^\alpha_{\rm a}(0,T)$, decomposing 
$v=a+bi$ for $a,b$ real-valued we can write
\[
\delta^2 \mathcal{E}(\phi)v=L_+a + iL_-b
\]
where the operators $L_{\pm}$ are linear operators acting on $L^2_{\rm a}([0,T];\RM)$
defined by
\begin{align}
L_+ &:= \Lambda^\alpha - \gamma(2\sigma+1)\phi^{2\sigma} + \omega  \label{D:Lplus} \\
L_- &:= \Lambda^\alpha -\gamma\phi^{2\sigma} + \omega.  \label{D:Lminus}
\end{align}
Consequently, we can consider $\delta^2 \mathcal{E}(\phi)$ as the matrix operator ${\rm diag}(L_+,L_-)$
acting on the product space $L^2_{\rm a}([0,T];\RM)^2$.  Concerning the spectrum of $\delta^2 \mathcal{E}(\phi)$, observe that $\delta^2 \mathcal{E}(\phi)$ is bounded below
and self-adjoint on $L^2_{\rm a}(0,T)$ with compactly embedding domain $H^{\alpha}_{\rm a}(0,T)$.  
Consequently, the spectrum of $\delta^2 \mathcal{E}(\phi)$, and hence of the operators $L_{\pm}$,
acting on $L^2_{\rm a}(0,T)$
is comprised of a countably infinite discrete set of real eigenvalues tending to real $+\infty$ with no finite accumulation point.  Key information in the stability analysis of $\phi$ now rests
on determining the number of negative $T$-antiperiodic eigenvalues of $\delta^2 \mathcal{E}(\phi)$, referred
to as the \emph{Morse index} of the operator $\delta^2 \mathcal{E}(\phi)$, as well
as a characterization of its $T$-antiperiodic kernel.

Note that, by the translation and phase invariance of \eqref{E:FNLSrot},
it is straightforward to verify that
\[
L_+\phi'=0\quad\textrm{and}\quad L_-\phi=0,
\]
and hence $\phi'$ and $\phi$ belong to the $T$-antiperiodic kernel of the $T$-periodic coefficient
operators $L_+$ and $L_-$, respectively.
In general, it is very difficult to determine whether or not these functions comprise
the entirety of the $T$-antiperiodic kernels.  Indeed, it is not difficult to construct examples
in the local case $\alpha=2$ where either $L_{\pm}$ have kernels with higher multiplicity.
In the antiperiodic case, the issue is even further complicated by the fact that standard Perron-Frobenius
arguments fail to characterize the ground state eigenvalues of $\delta^2 \mathcal{E}(\phi)$ on $L^2_{\rm a}(0,T)$.
Indeed, even in the local case $\alpha=2$ it is not difficult to construct examples where
the first antiperiodic eigenvalue of a Schr\"odinger operator with periodic potential has algebraic multiplicity
two\footnote{See, for example, Lemma 1 in \cite{BJK11} and Lemma 4.2 in \cite{MJ09} for criteria encountered in the analysis of a KdV-type equation.}.
Furthermore, determining the number of $T$-antiperiodic negative eigenvalues of $\delta^2 \mathcal{E}(\phi)$
is often handled by classical Sturm-Liouville type arguments.  Indeed, in the classical case $\alpha=2$ the fact
that both $\phi$ and $\phi'$ have roots in $[0,T)$ implies that $\lambda=0$ may be either the first
or second $T$-antiperiodic eigenvalue of $L_{\pm}$, and hence a-priori the operator $\delta^2 \mathcal{E}(\phi)$
may have at most \emph{two} negative $T$-antiperiodic eigenvalues, which is typically an unfavorable
energy configuration for orbital stability.  When $\alpha\in(1,2)$, classical Sturm-Liouville arguments
do not apply to the operators $L_{\pm}$ and hence new methods will be necessary to determine
the number of negative eigenvalues of $\delta^2 \mathcal{E}(\phi)$.  Nevertheless, the main result for this section is the following.

\begin{proposition}[Nondegeneracy \& Morse Index Bounds]\label{P:nondegeneracy}
Let $\alpha \in (1,2)$ and $\sigma>0$ in the defocusing ($\gamma=-1$) fNLS \eqref{E:FNLS1}.
Let $\phi(\cdot;\mu):=\phi(\cdot,c=0,\mu)\in H^{\alpha/2}_{\rm a}(0,T)$ be a real-valued local minimizer
of $\mathcal{H}$ over $H^{\alpha/2}_{\rm a}(0,T)$ subject to fixed $Q(u)=\mu>0$ and
$N(u)=0$, as constructed in Lemma \ref{L:c=0_prop}, and assume that $\phi$ and the 
associated Lagrange multiplier $\omega$ 
depend on $c$ in a $C^1$ manner near $c=0$.  Then the associated Hessian operator $\delta^2 \mathcal{E}(\phi)$ acting on $L^2_{\rm a}(0,T)$ is nondegenerate,
i.e.
\[
\ker(\delta^2 \mathcal{E}(\phi))={\rm span}\left\{\phi',i\phi\right\}
\]
and\footnote{Here and throughout, for a given linear operator $\mathcal{L}$ on $L^2_{\rm a}(0,T)$ we
denote the Morse index of $\mathcal{L}$ as 
$n_-(\mathcal{L}):=\#\left\{\lambda\in\sigma_{L^2_{\rm a}(0,T)}(\mathcal{L}):\lambda<0\right\}$.} 
$n_-(\delta^2 \mathcal{E}(\phi))=1$.  Specifically, the operators $L_{\pm}$ are nondegenerate acting on $L^2_{\rm a}(0,T)$ with
\[
\ker\left(L_+\right)={\rm span}\{\phi'\}\quad\textrm{and}\quad\ker(L_-)={\rm span}\{\phi\}
\]
and, further, we have $n_-(L_+)=0$ and $n_-(L_-)=1$.
\end{proposition}

\begin{remark}
In the above proposition, we must view the minimizer $\phi(\cdot;\mu)$ as a member of a larger family
of traveling waves $\phi(\cdot;c,\mu)$ with $\phi(\cdot;c=0,\mu)=\phi(\cdot;\mu)$.  
In the local case when $\alpha=2$, this embedding and the smoothness of $\phi$ and $\omega$ 
on $(c,\mu)$ hold trivially due to ODE techniques and the presence of an exact Galilean symmetry.  
Furthermore, stability results for solitary and periodic waves in the local context fundamentally 
rely on the ability to differentiate $\phi$ and $\omega$ with respect to parameters, allowing 
one to connect the geometry and smoothness of the manifold of solutions to the associated stability theory;
see, for example, the discussion immediately following the proof of Proposition \ref{P:nondegeneracy} below.
In this nonlocal context, however, we have yet to obtain this smoothness result, which seems akin to a 
local uniqueness result for the minimization problem in Proposition \ref{P:min}.  We consider
this an interesting direction for future research.  Since our methods, similarly to the analysis
in the local case $\alpha=2$, depend heavily
on $\phi$ and $\omega$ depending smoothly on $c$ near $c=0$, we will take this as an assumption
in the subsequent analysis;  see also
Remark \ref{r:travel} below.  Furthermore,
it is important to note that even in cases where such smooth dependence is a-priori given, nondegeneracy is far from a foregone conclusion
and still represents a formidable problem.   Finally, we note that while the smooth dependence of $\phi$ and $\omega$
on $\mu$ is not needed for this nondegeneracy theory, it will play a role in the forthcoming stability theory.  See Remark \ref{r:mu_stability} below.
\end{remark}

\begin{remark}
At this point, the reason for our restriction to antiperiodic perturbations begins to become clear.  If we would consider the stability of a $T$-antiperiodic wave $\phi(\cdot,c=0,\mu)$
to the more ``natural" class of co-periodic, i.e.  $2T$-periodic, perturbations the forthcoming analysis then shows that $L_{\pm}$ would then 
each have \emph{at least} one negative eigenvalue.  Indeed, by Lemma \ref{L:properties_of_Kp} below it follows that the ground state $2T$-periodic eigenvalues of $L_{\pm}$ must be
sign-definite.  Since $\phi'$ and $\phi$ belong to the kernels of $L_+$ and $L_-$, respectively, the fact that neither of these functions are sign-definite imply that both operators
$L_{\pm}$ must have at least one negative eigenvalue, violating the structural hypotheses necessary for the application of the stability theory in \cite{GSS1,GSS2}: see 
Remark \ref{R:coperiodic_groundstate} below.
In the case $\alpha=2$, this was previously observed in  \cite{Pava07}, where the author found that $L_+$ has at least two negative co-periodic eigenvalues, while $L_-$ has at least
one.  
\end{remark}

As noted in the introduction, Proposition \ref{P:nondegeneracy} was established using ODE techniques in the local
case $\alpha=2$ by Gallay and Haragus \cite{GH07}.  Precisely, their proof utilizes Sturm-Liouville theory
for (local) differential operators, together with a homotopy argument and
a-priori control over the dimension of the $T$-antiperiodic
kernels.  While these ODE-based techniques are not directly available in the nonlocal setting $\alpha\in(0,2)$,
we recall that Frank and Lenzmann \cite{FL13} recently obtained the nondegeneracy of the linearization about 
solitary waves for a family of nonlinear nonlocal models that include the \emph{focusing} ($\gamma=+1$) fNLS
\eqref{E:FNLS1}.  Their idea was to find a suitable substitute for the Sturm-Liouville oscillation theory to control
the number of sign changes in eigenfunctions for a fractional Schr\"odinger operator with real, localized potential.
This theory, developed on the line, was then adapted to the periodic setting in \cite{HJ15}, where the authors
considered the nonlinear orbital stability of $T$-periodic traveling wave solutions to the fractional KdV 
equation.  

The proof of Proposition \ref{P:nondegeneracy} extends these previous nondegeneracy results to encompass
the \emph{$T$-antiperiodic spectra} of fractional Schr\"odinger operators with real, $T$-periodic potentials.
As mentioned above, this extension is significant as, even in the classical $\alpha=2$ case, the ground state
antiperiodic eigenvalue of $T$-periodic linear Schr\"odinger type operators need not be simple.
This is in stark contrast to the ground state $T$-periodic eigenvalues of such operators, which
are \emph{always simple} by Perron-Frobenius theory.  While our proof follows the basic strategy
in \cite{FL13} and \cite{HJ15}, substantial modifications are necessary to accommodate
the antiperiodic structure of the admissible class of perturbations.

\

There are two key analytical results necessary to establish Proposition \ref{P:nondegeneracy}.
First, we require an appropriate characterization of the ground state eigenfunctions
of $L_{\pm}$ acting on $L^2_{\rm a}(0,T)$.  A natural approach is to attempt
to use a Perron-Frobenius type argument, demonstrating that the semigroups $e^{-L_{\pm}t}$ are positivity
improving on appropriate subspaces of $L^2_{\rm a}(0,T)$.  Second, we require a nonlocal 
Sturm-Liouville type oscillation theory for the second antiperiodic eigenfunctions of $L_{\pm}$.  
Following the general ideas in \cite{FL13} and \cite{HJ15}, this is accomplished by extending
the antiperiodic eigenvalue problems for $L_{\pm}$ on $L^2_{\rm a}(0,T)$ to appropriate local problems
on the upper half-space.

\subsection{Perron-Frobenius Theory for Antiperiodic Eigenfunctions}  \label{S:GS_Theory}

The goal of this section is to provide a characterization of the antiperiodic ground state
eigenfunctions for linear, fractional Schr\"odinger operators of the form
\begin{equation}\label{E:op}
L := \Lambda^\alpha + V(x),
\end{equation}
where the potential $V(x)$ is even, real-valued, smooth and $T$-periodic for some finite $T > 0$.  In particular,
we will classify properties of the $T$-antiperiodic ground state for $L$, along with upper bounds
on the number sign changes on higher $T$-antiperiodic eigenfunctions.  As noted in the introduction, 
even in the local case $\alpha=2$ such results are nontrivial  as $T$-antiperiodic ground
states need not be simple.  As we will see below, this comes from the fact that the semigroup
generated by $L$ is not positivity improving when acting on $L^2_{\rm a}(0,T)$.  To handle
this difficulty, we will decompose the space $L^2_{\rm a}(0,T)$ of real-valued 
$T$-antiperiodic functions on $\mathbb{R}$
into the (invariant) even and odd subspaces, and develop ground state and oscillation theories for the
operator $L$ in each subspace separately.  As we will see, restricted to these subspaces, the semigroup
generated by $L$ will indeed be positivity improving.  Finally, using rearrangement properties we find
an ordering between the antiperiodic odd and even ground state eigenvalues for $L$ in terms of monotonicity
properties of the potential $V$ on $(0,T)$.

We begin by observing that the $T$-periodicity of the potential $V$ implies that 
the operator $L$ is well defined as a closed, densely defined operator from $L^2_{\rm a}(0,T)$ into itself.
Since $V$ is a bounded and smooth potential, the operator $L$ is a relatively
compact perturbation of the operator $-\Lambda^\alpha$, 
Theorem XIII.44 from \cite{ReedSimonIV} implies
that the ground state eigenvalues of $L$ acting on an invariant subspace $\mathcal{Y}$ of $L^2_{\rm a}(0,T)$
is simple as an eigenvalue of $L|_{\mathcal{Y}}$ provided the fractional heat semigroup
$\left\{e^{-\Lambda^\alpha t}\right\}_{t\geq 0}$ is positivity improving on $\mathcal{Y}$; that is, if
\[
f\in\mathcal{Y},~~f\geq 0,~~f\neq 0\quad\implies\quad e^{-\Lambda^\alpha t}f>0\textrm{~~on~~}\mathcal{Y}.
\]
Thus, it is sufficient to study the semigroup generated by $-\Lambda^\alpha$ on $L^2_{\rm a}(0,T)$, which
we shall study below by first considering the semigroup acting on $L^2(\mathbb{R})$ and then
considering appropriate periodizations of its integral kernel.

\

The semigroup $e^{-\Lambda^\alpha t}$ acting on $L^2(\RM)$ is naturally understood via the Fourier
transform.  Throughout, the operator $\mathcal{F}$ will denote the extension to the 
space of tempered distributions $\mathcal{S}'(\mathbb{R})$ of the Fourier transform
\[
\mathcal{F}(f)(\xi):=\frac{1}{\sqrt{2\pi}}\int_{\mathbb{R}}e^{-i\xi x}f(x)dx
\]
on the Schwartz space $\mathcal{S}(\mathbb{R})$, with inverse 
$\mathcal{F}^{-1}(f)(x):=\frac{1}{\sqrt{2\pi}}\int_{\mathbb{R}}e^{i\xi x}f(\xi)d\xi$. In particular,
with this particular normalization $\mathcal{F}$ defines a unitary operator on $L^2(\mathbb{R})$.
For all $t\geq 0$, the operators $e^{-\Lambda^\alpha t}$ acting on $\mathcal{S}(\mathbb{R})$ can be understood
via
\begin{equation}  \label{D:semigroup}
e^{-\Lambda^\alpha t}f(x) = \CalF^{-1}\parens{e^{-\abs{\cdot}^\alpha t} \widehat{f}(\cdot)}(x) = \frac{1}{\sqrt{2\pi}} \int_\R e^{-\abs{\xi}^\alpha t} \widehat{f}(\xi) e^{i\xi x} \, d\xi.
\end{equation}
Alternatively, introducing the function
\[
K(x,t) := \CalF^{-1}\parens{e^{-\abs{\cdot}^\alpha t}}(x)
\]
we see that $e^{-\Lambda^\alpha t}$ can be viewed as the convolution operator
\[
e^{-\Lambda^\alpha t}f(x) = \int_\R K(x-y,t) f(y) \, dy.
\]
Note that, when $\alpha=2$, the integral kernel $K$ agrees with the standard heat-kernel and can be
explicitly expressed in terms of a Gaussian function.  While such explicit formulas 
are not available in the nonlocal case $\alpha\in(0,2)$, in the recent work of Frank \& Lenzmann 
\cite[Appendix A]{FL13} it was observed that, for all $t>0$ and $\alpha\in(0,2)$, the kernel
$K(\cdot,t)$ is even and strictly positive with $\partial_xK(x,t)<0$ for all $x>0$
and, furthermore, decays rapidly at spatial infinity.
Further, we know that $K(\cdot,t) \in L^1(\R)$ since, by the positivity of $K$,
\[
\norm{K(\cdot,t)}_{L^1(\R)} = \int_\R K(x,t) \, dx = \CalF(K(\cdot,t))(\xi=0) = 1.
\]

Since our interest is in a $T$-antiperiodic oscillation theory for operators of 
the form \eqref{E:op} we now describe how $e^{-\Lambda^\alpha t}$ acts on periodic
functions.  Since $K(\cdot,t)$ lies in $L^1(\mathbb{R})$ for all $t>0$, given
any $f\in L^\infty(\mathbb{R})$ that is $2T$-periodic we can write
\begin{align*}
e^{-\Lambda^\alpha t} f(x) 
&= \int_{-T}^T \parens{\sum_{n \in \Z} K(x-y+2nT,t)}f(y) \, dy 
= \int_{-T}^T K_p(x-y,t)f(y) \, dy,
\end{align*}
where
\begin{equation}\label{E:Kp_rep1}
K_p(x,t) := \sum_{n \in \Z} K(x+2nT,t)
\end{equation}
represents the $2T$-periodic \emph{periodization} of the integral kernel $K$.  Observe
that the sum defining $K_p$ is absolutely convergent for each $t>0$ due to the rapid
decay of $K(\cdot,t)$ at spatial infinity.  Furthermore, one can show that $K_p\in L^1_{\rm per}(0,2T)$
and admits the Fourier series expansion
\[
K_p(x,t) = \sum_{n \in \Z} e^{-\left|\pi n / T \right|^\alpha t} e^{\pi i n x / T}
\]
so that, by the convolution theorem, we have that $e^{-\Lambda^\alpha t}$ acts on $2T$-periodic
functions via
\begin{equation}\label{E:eLt_fourier}
e^{-\Lambda^\alpha t} f(x) = \sum_{n \in \Z} e^{-\left| \pi n / T \right|^\alpha t} \hat{f}(n) e^{\pi i n x / T}
\end{equation}
In particular, the operator $e^{-\Lambda^\alpha t}$ on $2T$-periodic functions has the exact same
Fourier symbol as when considering the same operator on $L^2(\mathbb{R})$.

While the Fourier representation \eqref{E:eLt_fourier} seems useful for numerical calculations,
we are unfortunately unable to extract from it the necessary information for our forthcoming
theory.  Nevertheless, using a slightly different representation coming from Bernstein's theorem,
we have the following.

\begin{lemma}  \label{L:properties_of_Kp}
For all $t > 0$ and $\alpha\in(0,2]$, $K_p(\cdot,t)$ is positive, even, 
$2T$-periodic, and strictly decreasing on $[0,T]$.
\end{lemma}

\begin{remark} It was recently shown in \cite{EWhighest} that the periodization of a function $f\in L^1(\RM)$
that is even and completely monotone\footnote{That is, $(-1)^j\partial_z^{j}g(z)\geq 0$ for all $j\in\mathbb{N}$ and $z>0$.} on $(0,\infty)$ is automatically even and completely monotone on a
half period.  However, as is evident from the subordination formula \eqref{E:subordination} below, the
kernel $K(x)$ is \emph{not} completely monotone on $(0,\infty)$, and hence this abstract result
does not apply.
\end{remark}

\begin{remark}\label{R:coperiodic_groundstate}
By the result in  \cite{ReedSimonIV} discussed above, Lemma \ref{L:properties_of_Kp} implies that the $2T$-periodic ground state eigenvalues for the operators $L_{\pm}$ must be simple
with sign-definite eigenfunctions.  In particular, since $\phi'$ and $\phi$ belong to $2T$-periodic kernels of $L_+$ and $L_-$, respectively, the fact neither $\phi'$ nor $\phi$ are sign
definite immediately implies that both $L_{\pm}$ have at least one negative eigenvalue, invalidating the stability hypotheses of \cite{GSS1,GSS2}.  This result emphasizes why
we restrict to $T$-antiperiodic perturbations in our work.
\end{remark}

\begin{proof}
The whole line kernel $K(x,t)$ was shown by Frank \& Lenzmann \cite[Appendix A]{FL13} to be even and
positive for all $t>0$, $x\in\RM$.  Since $K(\cdot,t)\in L^1(\mathbb{R})$ for all $t>0$, the representation
\eqref{E:Kp_rep1} implies that the periodization $K_p(\cdot,t)$ must also be even, positive and
$2T$-periodic for all $t>0$.
To prove that $K_p(\cdot,t)$ is decreasing on $(0,T)$ for each $t>0$, 
we follow \cite{FL13} and observe that 
the function $g(z)=e^{-z^{\alpha/2}}$ is completely monotone on the positive half-line $(0,\infty)$ 
for all $0<\alpha\leq 2$ and hence, by Bernstein's theorem, is the Laplace
transform of a non-negative finite measure $\nu_\alpha$ depending on $\alpha$, i.e. 
$e^{-z^{\alpha/2}}=\int_0^\infty e^{-\tau z}d\nu_\alpha(\tau)$ for some such measure $\nu_\alpha$.
Setting $z=|x|^2$ and recalling the inverse Fourier representation for 
the Gaussian $e^{-\tau\xi^2}$ leads to\footnote{The subordination formula
in \cite{FL13} is stated only for the case $t=1$.  This more general formula follows from the
scaling $K(x,t) = t^{-1/\alpha} K(t^{-1/\alpha} x, 1)$.} the ``subordination formula"
\begin{equation}\label{E:subordination}
K(x,t) = t^{-1/\alpha} \int_0^\infty \frac{1}{\sqrt{2\tau}} \exp\parens{-\frac{t^{-2/\alpha} x^2}{4\tau}} \, 
d\nu_\alpha(\tau)
\end{equation}
valid for all $x \in \R$ and $t > 0$.  From \eqref{E:Kp_rep1} it follows
that for all $\alpha\in(0,2)$ the $2T$-periodic kernel $K_p$ can be expressed as
\begin{align*}
K_p(x,t) 
&= t^{-1/\alpha} \int_0^\infty \frac{1}{\sqrt{2\tau}} \brackets{ \sum_{n \in \Z} \exp\parens{-\frac{(x+2nT)^2}{4 t^{2/\alpha} \tau}} } d\nu_\alpha(\tau) \\
&= t^{-2/\alpha} \sqrt{2\pi} \int_0^\infty \brackets{ \sum_{n \in \Z } \frac{1}{\sqrt{4\pi u}} \exp\parens{-\frac{(x+2nT)^2}{4u}} } d\nu_\alpha(u),
\end{align*}
where the final equality follows from the variable substitution $u=t^{2/\alpha}\tau$.
The integrand above may be recognized as the $2T$-periodized Gauss-Weierstrass kernel
\begin{equation}  \label{E:gauss-weierstrass}
\vartheta_u(x) := \sum_{n \in \Z } \frac{1}{\sqrt{4\pi u}} \exp\parens{-\frac{(x+2nT)^2}{4u}},
\end{equation}
hence we may express $K_p(x,t)$ compactly as
\[
K_p(x,t) = t^{-2/\alpha} \sqrt{2\pi} \int_0^\infty \vartheta_u(x) \, d\nu_\alpha(u).
\]
The monotonicity properties of the function $\vartheta_u(x)$ have been
studied in \cite[Theorem 4.2]{Andersson13}, where it was shown\footnote{While the results in \cite{Andersson13}
were stated only for the case $T=\pi$ they easily extend to this more general setting via scaling.} to be strictly decreasing in $x$ on $(0,T)$ for all $u > 0$.  It now clearly follows that if $x,y \in (0,T)$ with $x < y$, then
\[
K_p(x,t) - K_p(y,t) = t^{-2/\alpha} \sqrt{2\pi} \int_0^\infty \parens{ \vartheta_u(x) - \vartheta_u(y) } \, 
d\nu_\alpha(u) > 0,
\]
i.e. $K_p(\cdot,t)$ is decreasing on $(0,T)$ for all $t > 0$, as claimed.
\end{proof}

\begin{remark}
A consequence of Bernstein's theorem for completely-monotone functions, the subordination formula \eqref{E:subordination} conveniently encodes all $\alpha$-dependence into the non-negative measure $\nu_\alpha$, which facilitates studying the fractional heat kernel using familiar techniques of the classical heat kernel.  In order for this result to apply, we must restrict to $\alpha\in(0,2]$ as the function $g(z) = \exp(-z^{\alpha/2})$ 
fails to be completely monotone for $\alpha > 2$.
\end{remark}

To study the antiperiodic eigenvalues of $L$, we now 
further restrict the semigroup $e^{-\Lambda^\alpha t}$ to the subspace $L^2_{\rm a}(0,T)$ of $T$-antiperiodic
functions.  For such $f\in L^2_{\rm a}(0,T)$ we have the representation
\[
e^{-\Lambda^\alpha t}f(x) 
= \int_0^T \brackets{ K_p(x-y,t) - K_p(x-y-T,t)}f(y) \, dy 
=:\int_0^T K_a(x-y,t)f(y) \, dy
\]
for the action of $e^{-\Lambda^\alpha t}$ on $T$-antiperiodic functions, where $K_a$
denotes the $T$-antiperiodic kernel
\begin{equation}  \label{D:K_a}
K_a(x,t) := K_p(x,t) - K_p(x-T,t).
\end{equation}
Next, we gather some important properties of $K_a$.

\begin{lemma}[Properties of $K_a$]\label{L:properties_of_Ka}
For all $t>0$ and $\alpha\in(0,2]$, the function $K_a(\cdot,t)$ is even, $T$-antiperiodic and 
strictly positive for all $x\in(-T/2,T/2)$.  Furthermore, $K_a(\cdot,t)$ is odd about $x=T/2$
and is strictly decreasing on $(0,T)$.
\end{lemma}
\begin{proof}
The parity and antiperiodicity of $K_a$ follow directly from \eqref{D:K_a}.
Since all even, $T$-antiperiodic functions are odd\footnote{Indeed, if $f$ is even and $T$-antiperiodic
then $f(x+T/2)=f(-x-T/2)=f(-x+T/2)$.} about $x=T/2$, it remains to show
that $K_a(\cdot,t)$ is positive on $(-T/2,T/2)$ and strictly decreasing on $(0,T)$.
To this end, fix $x\in(0,T/2)$ and observe that the evenness of $K_p(\cdot,t)$ implies
that
\begin{equation}\label{E:Kp_rep2}
K_a(x,t) = K_p(x,t) - K_p(T-x,t) > 0, 
\end{equation}
where the strict inequality follows since $K_p(\cdot,t)$ is strictly decreasing on $(0,T)$ 
by Lemma \ref{L:properties_of_Kp}
and $0<x<T-x<T$ for $x\in(0,T/2)$.
Since $K_a(\cdot,t)$ is even for all $t>0$, the positivity of $K_a(\cdot,t)$ on $(-T/2,T/2)$ follows.
Similarly, differentiating \eqref{E:Kp_rep2} with respect to $x$, it follows
that for $x\in(0,T)$ we have
\[
\partial_x K_a(x,t) = \partial_x K_p(x,t) + \partial_x K_p(T-x,t)<0
\]
where we have used that $\partial_x K_p(x,t)<0$ for all $x\in(0,T)$ by Lemma \ref{L:properties_of_Kp}.
\end{proof}

An important consequence of Lemma \ref{L:properties_of_Ka} is that the semigroup $e^{-\Lambda^\alpha t}$
is not positivity improving (nor even positivity preserving) on $L^2_{\rm a}(0,T)$, and hence, by the functional
calculus, the antiperiodic ground states of the operator $L$ in \eqref{E:op} cannot be
characterized by standard Perron-Frobenius arguments.  However, we note that since
the potential $V(x)$ in \eqref{E:op} is even, the operator $L$ respects the orthogonal
decomposition
\[
L^2_{\rm a}(0,T)=L^2_{\rm a,even}(0,T)\oplus L^2_{\rm a,odd}(0,T),
\]
where $L^2_{\rm a,even/odd}(0,T)$ denotes the subspaces of even/odd functions in $L^2_{\rm a}(0,T)$, respectively.
Precisely, the subspaces $L_{\rm a,even/odd}(0,T)$ are invariant subspaces for $L$ and the above decomposition
implies that all eigenfunctions for $L$ may be chosen to be either even or odd.  In particular,
\[
\sigma_{L^2_{\rm a}(0,T)}\left(L\right)=\sigma_{L^2_{\rm a,even}(0,T)}\left(L\right) \,
			\bigcup \, \sigma_{L^2_{\rm a,odd}(0,T)}\left(L\right),
\]
where we emphasize the above spectral decomposition need not  be disjoint.  Next, we consider
the action of the semigroup $e^{-\Lambda^\alpha t}$ on the above invariant subspaces.

First, note that if $f\in L^2_{\rm a,even}(0,T)$ then
\begin{align*}
e^{-\Lambda^\alpha t}f(x) &= \frac{1}{2} \brackets{ \int_0^T K_a(x-y,t)f(y) \, dy + \int_{-T}^0 K_a(x+y,t)f(y)dy}\\
&= \frac{1}{2} \brackets{ \int_0^T K_a(x-y,t)f(y) \, dy + \int_0^T K_a(x+y-T,t)f(y-T) \, dy } \nonumber \\
&= \frac{1}{2} \int_0^T \brackets{ K_a(x-y,t) + K_a(x+y,t) }f(y) \, dy  \nonumber 
\end{align*}
where the final equality follows from the $T$-antiperiodicity of both $K_a(\cdot,t)$ and $f$.
Observe that since $f(y)$ and $K_a(x-y,t)+K_a(x+y,t)$ are even and $T$-antiperiodic in $y$, 
they are both odd functions in $y$ about $y=T/2$.  Consequently, their product is even
in $y$ about $y=T/2$, which yields the representation
\begin{equation}  \label{E:semigroup_even}
e^{-\Lambda^\alpha t}f(x) = \int_0^{T/2} \brackets{ K_a(x-y,t) + K_a(x+y,t) }f(y) \, dy
\end{equation}
for the action of semigroup $e^{-\Lambda^\alpha t}$ on $L^2_{\rm a,even}(0,T)$.

\begin{lemma}  \label{L:even_heat_kernel_positivity}
For all $x,y\in(-T/2,T/2)$ and $t>0$, we have
\[
K_a(x-y,t) + K_a(x+y,t) > 0.
\]
In particular, the semigroup $e^{-\Lambda^\alpha t}$ restricted to $L^2_{\rm a,even}(0,T)$ is positivity
improving, i.e. if $f\in L^2_{\rm a,even}(0,T)$ is non-trivial with $f(x)\geq0$ for $x\in(-T/2,T/2)$, then
$e^{-\Lambda^\alpha t}f(x)>0$ for all $x\in(-T/2,T/2)$.
\end{lemma}

\begin{proof}
We begin by proving the claim for $x,y \in (0,T/2)$.  Fix $t > 0$ and define $G(x,y;t) := K_a(x-y,t) + K_a(x+y,t)$, and note that $G(x,y;t) = G(y,x;t)$ for all $x,y$.  So, without loss of generality we need only prove that $G(x,y;t) > 0$ for all $(x,y) \in \mathcal{R} := \{ (x,y) : 0 < x < T/2, \ 0 < y \leq x \}$.  Observe that for all $(x,y) \in \mathcal{R}$, we have
$$
0 \leq x-y < T/2 \quad \text{and} \quad 0 \leq x+y < T,
$$
hence
$$
\partial_x G(x,y;t) = \partial_x K_a(x+y,t) + \partial_x K_a(x-y,t) < 0
$$
since $K_a(\cdot,t)$ is decreasing on $(0,T)$ by Lemma \ref{L:properties_of_Ka}.  Moreover, for all $y \in (0,T/2)$, we have
$$
G(T/2,y;t) = K_a(T/2+y,t) + K_a(T/2-y,t) = 0
$$
since $K_a(\cdot,t)$ is odd about $T/2$, again by Lemma \ref{L:properties_of_Ka}.  Thus for every $y_0 \in (0,T/2)$, the function $x \mapsto G(x,y_0;t)$ is decreasing on $y_0 < x < T/2$ toward the value $G(T/2,y_0;t) = 0$, hence it must be that $G(x,y;t) > 0$ for all $(x,y) \in \mathcal{R}$, and we conclude that $G(x,y;t) > 0$ for all $x,y \in (0,T/2)$.  Finally, since $G(x,y;t) > 0$ for all $x,y \in (0,T/2)$, we also have that $G(x,y;t) > 0$ for all $x,y \in (-T/2,T/2)$ since $G$ is invariant under the maps $x \mapsto -x$ and $y \mapsto -y$.
\end{proof}

\begin{remark}
Alternatively to the above proof, one can observe that for each fixed $t>0$, the function
$G(x,y;t):=K_a(x-y,t)+K_a(x+y,t)$ is a solution of the IVBVP
\[\left\{\begin{aligned}
&G_{yy}=G_{xx},~~x,y\in(-T/2,T/2)\\
&G(x,0;t)=2K_a(x,t),~~x\in(-T/2,T/2)\\
&G(\pm T/2,y;t)=0,~~y\in(-T/2,T/2),
\end{aligned}\right.
\]
which is simply the 1D-wave equation (here treating $y$ as the temporal variable) 
with homogeneous Dirichlet boundary conditions.  Since $K_a(x,t)>0$
for $x\in(-T/2,T/2)$, the positivity of $G(x,y;t)$ for $x,y\in(-T/2,T/2)$ and $t>0$ follows immediately.
\end{remark}

Turning our attention to the odd subspace, similar calculations to those above
yield the representation
\begin{equation}\label{E:semigroup_odd}
e^{-\Lambda^\alpha t} f(x) = \frac{1}{2} \int_0^T \brackets{ K_a(x-y,t) - K_a(x+y,t) } f(y) \, dy.
\end{equation}
for the action of the semigroup $e^{-\Lambda^\alpha t}$ on $L^2_{\rm a,odd}(0,T)$.  

\begin{lemma}  \label{L:odd_heat_kernel_positivity}
For all $t>0$ and $x,y\in(0,T)$, we have
\[
K_a(x-y,t) - K_a(x+y,t) > 0.
\]
In particular, the semigroup $e^{-\Lambda^\alpha t}$ restricted to $L^2_{\rm a,odd}(0,T)$ is positivity improving,
i.e. if $f\in L^2_{\rm a,odd}(0,T)$ is nontrivial with $f(x)\geq 0$ for $x\in(0,T)$, then
$e^{-\Lambda^\alpha t}f(x)>0$ for all $x\in(0,T)$.
\end{lemma}

\begin{proof}
Fix $t>0$ and $x,y \in (0,T)$, and observe that the $T$-antiperiodicity of $K_a(\cdot,t)$ implies
\begin{align*}
K_a(x-y,t) - K_a(x+y,t) &= K_a(x-y,t) + K_a(x+y-T,t)\\ 
&= K_a\left( \left(x-\frac{T}{2}\right) - \left(y - \frac{T}{2}\right)\right)  + K_a\parens{ \parens{x - \frac{T}{2}} + \parens{y - \frac{T}{2}} }.
\end{align*}
Since $x-T/2,~y-T/2\in(-T/2,T/2)$, the proof follows by Lemma \ref{L:even_heat_kernel_positivity}.
\end{proof}

From Lemma \ref{L:even_heat_kernel_positivity} and Lemma \ref{L:odd_heat_kernel_positivity}, the ground state
eigenfunctions of $e^{-\Lambda^\alpha t}$ acting on the invariant subspace $L^2_{\rm a,even/odd}(0,T)$
are positivity improving.  Since the operator $L$
defined in \eqref{E:op} is a relatively compact perturbation of $-\Lambda^\alpha$, we can apply
standard Perron-Frobenius arguments to deduce that the largest eigenvalues of $e^{-Lt}$
restricted to $H^{\alpha/2}_{\rm a,even}(0,T)$ and $H^{\alpha/2}_{\rm a,odd}(0,T)$ separately
are simple with strictly positive eigenfunction on $(-T/2,T/2)$ and $(0,T)$, respectively;
see \cite[Theorem XIII.44]{ReedSimonIV}, for instance.  By the functional calculus, this establishes
the following characterization of the antiperiodic ground states.

\begin{theorem}[Antiperiodic Ground State Theory]\label{T:ground}
Let $\alpha\in(1,2)$ and let $V:\mathbb{R}\to\mathbb{R}$ be an even, smooth, $T$-periodic potential and consider the linear operator
$L=\Lambda^\alpha+V(x)$ acting on $L^2_{\rm a}(0,T)$.  
\begin{itemize}
\item[(a)] The ground state eigenvalue of $L$ 
restricted to $L^2_{\rm a,even}(0,T)$ is simple, and the 
corresponding $T$-antiperiodic, even eigenfunction is sign-definite on $(-T/2,T/2)$.  
\item[(b)] The ground state eigenvalue of $L$ restricted to $L^2_{\rm a,odd}(0,T)$ is simple, 
and the corresponding $T$-antiperiodic, odd eigenfunction is sign-definite on $(0,T)$.
\end{itemize}
\end{theorem} 

\begin{proof}
Parts (a) and (b) follow directly from Lemma \ref{L:even_heat_kernel_positivity},
Lemma \ref{L:odd_heat_kernel_positivity}, and  \cite[Theorem XIII.44]{ReedSimonIV}, as discussed above.
\end{proof}

While Theorem \ref{T:ground} establishes the simplicity of the even and odd antiperiodic ground state
eigenvalues of $L$ on the respective subspace, it is natural to consider the \emph{ordering} between
these ground state eigenvalues. 
When the potential has sufficiently small amplitude, the ordering between odd and even $T$-antiperiodic ground 
state eigenvalues may be verified directly through the use of bifurcation theory:
see, for example, Proposition 6.2 and Remark 6.3 in \cite{NPS} where the analysis was carried out in a local
context.  In that case, the ground state eigenvalues agree at zero-amplitude and one tracks the splitting of these
eigenvalues for very small amplitudes.  
For general amplitude potentials, however, in the local case
$\alpha=2$ it was shown in \cite[Lemma 2.2]{DPR} through the use ODE techniques and increasing/decreasing rearrangement
inequalities that the ordering of these ground states 
depends sensitively on the monotonicity properties of the periodic potential $V$ in \eqref{E:op}. 
Using symmetric antiperiodic rearrangement inequalities,
together with the above nonlocal ground state theory, we are able to extend the results of \cite{DPR}
to the nonlocal setting $\alpha\in(1,2)$; see Appendix \ref{A:Polya}.  
Such information will be used heavily in the coming sections.

\begin{proposition}[Ground State Ordering]\label{P:gsordering}
Let $\alpha\in(1,2)$ and let $V:\RM\to\RM$ be an even, smooth, $T$-periodic potential, and consider
the linear operator $L=\Lambda^\alpha+V(x)$ acting on $L^2_{\rm a}(0,T)$.
\begin{itemize}
\item[(i)] If the potential $V$ is nonincreasing on $(0,T/2)$, then the ground state $T$-antiperiodic
eigenvalue of $L$ has at least one odd eigenfunction, i.e.
\[
\min\sigma\left(L\big{|}_{L^2_{\rm a,odd}(0,T)}\right)\leq\min\sigma\left(L\big{|}_{L^2_{\rm a,even}(0,T)}\right).
\]
\item[(ii)] If the potential $V$ is nondecreasing on $(0,T/2)$, then the ground state $T$-antiperiodic
eigenvalue of $L$ has at least one even eigenfunction, i.e.
\[
\min\sigma\left(L\big{|}_{L^2_{\rm a,even}(0,T)}\right)\leq\min\sigma\left(L\big{|}_{L^2_{\rm a,odd}(0,T)}\right).
\]
\end{itemize}
\end{proposition}

\subsection{Antiperiodic Oscillation Theory}

In addition to the above ground state theories, we require
a Sturm-Liouville type oscillation theory to characterize the possible nodal patterns for the second antiperiodic eigenfunctions of the fractional Schr\"odinger operators $L_{\pm}$.  To this end,
first note that an $H^{\alpha/2}_{\rm a}(0,T)$-eigenfunction of $L_{\pm}$ is necessarily continuous, bounded,
and can be chosen to be real-valued.  Following the ideas in \cite{FL13} and \cite{HJ15}, we proceed
by extending the eigenvalue problem associated to $L_{\pm}$ on $L^2_{\rm a}(0,T)$
to an appropriate local problem in the upper half space.

Note that the operator $\Lambda^\alpha$ acting on $L^2_{\rm a}(0,T)$ can be viewed as the
Dirichlet-to-Neumann operator for a suitable \emph{local} problem in the antiperiodic
half-strip $[0,T]\times(0,\infty)$.  Indeed, following \cite{CS07,RS15}, for a given $\alpha\in(0,2)$,
there exists a constant $C(\alpha)$ such that for any $f\in H^{\alpha}_{\rm a}(0,T)$
we have
\[
C(\alpha)\Lambda^\alpha f:=\lim_{y\to 0^+}y^{1-\alpha}w_y(\cdot,y),
\]
where $w=:E(f)\in C^\infty((0,\infty);H^{\alpha/2}_{\rm a}(0,T))\cap C([0,\infty);L^2_{\rm a}(0,T))$
is the unique solution to the elliptic boundary value problem
\[
\left\{\begin{array}{lcl}
\Delta w+\frac{1-\alpha}{y}w_y=0, & {\rm in } & [0,T]_{\rm antiper}\times(0,\infty)\\
w=f & {\rm on } & [0,T]_{\rm antiper}\times\{0\}
\end{array}\right.
\]
As in \cite{FL13,HJ15}, it follows that the eigenvalue problems 
for $L_{\pm}$ on $L^2_{\rm a}(0,T)$ can be extended
to an eigenvalue problem for a \emph{local} elliptic problem in the antiperiodic upper-half space
$[0,T]_{\rm antiper}\times(0,\infty)$ and, as such, one may derive a variational characterization
for the $T$-antiperiodic eigenvalues and eigenfunctions of $L_{\pm}$.  

If $v\in L^2_{\rm a}(0,T)$ is an eigenfunction associated to $L_+$, say,
then the extension $E(v)$ belongs to $C^0([0,T]_{\rm antiper}\times[0,\infty))$.  Defining
the zero set of $v$ to be 
\[
\mathcal{N}:=\left\{(x,y)\in[0,T]_{\rm antiper}\times[0,\infty):E(v)(x,y)=0\right\},
\]
which is clearly closed in $[0,T]_{\rm antiper}\times[0,\infty)$, we define the
\emph{nodal domains} of $E(v)$ to be the connected components of the open
set $\left([0,T]_{\rm antiper}\times[0,\infty)\right)\setminus\mathcal{N}$.  Recalling
the classical Courant nodal domain theorems yield an upper bound for the number
of nodal domains of $E(v)$ in $[0,T]_{\rm antiper}\times(0,\infty)$, we find the following oscillation result.

\begin{lemma}[Antiperiodic Oscillation Theory]\label{L:oscillation}
Under the hypothesis of Theorem \ref{T:ground}, any even (resp. odd) 
$T$-antiperiodic eigenfunctions of $L$
associated with the second eigenvalue (not counting multiplicity\footnote{That is, 
only the distinct elements of the $T$-antiperiodic spectrum of $L$ are listed.}) has at most two sign
changes over $(-T/2,T/2)$ (resp. $(0,T)$).
\end{lemma}

\begin{proof}
The proof follows along the same lines as \cite[Lemma 3.2]{HJ15} and
\cite[Theorem 3.1]{FL13}, and hence we only sketch the details here; see also \cite{HJM17}.  Consider
the spectrum of $L$ acting on $L^2_{\rm a,even}(0,T)$.  Suppose that
$v(x)$ is an even $T$-antiperiodic eigenfunction of $L$ associated with the its second
eigenvalue $\Lambda_2$, and suppose that $v$ has at least three sign changes in $(-T/2,T/2)$.  Clearly,
since $v$ is even, it follows that $v$ actually has at least four sign changes in $(-T/2,T/2)$,
and hence there are points
\[
-T/2<x_1<y_1<x_2<y_2<x_3<T/2
\]
such that, up to switching signs, $v(x_j)>0$ and $v(y_j)<0$.  Now, using the aforementioned
variational characterization of $\Lambda_2$, it follows from a standard Courant
nodal domain argument that the extension $E(v)$ can have at most two nodal domains in 
the strip $(-T/2,T/2)\times(0,\infty)$.  Since the nodal domains are open and connected, thus
pathwise connected, in $(-T/2,T/2)\times(0,\infty)$, we may find continuous curves 
$\gamma_\pm\in C^0([0,1];[-T/2,T/2)\times[0,\infty))$ such that
\[
\gamma_+(0)=x_1,\quad\gamma_+(1)=x_2,\quad\gamma_-(0)=y_1,\quad\gamma_-(1)=y_2
\]
and
\[
E(v)(\gamma_+(t))>0,\quad E(v)(\gamma_-(t))<0\quad{\rm for~ all}~~t\in[0,1].
\]
In particular, $\gamma_+(t)$ belongs to the same nodal domain for all $t\in(0,1)$, denoted $\Omega_+$,
while $\gamma_-(t)$ belongs to the same nodal domain $\Omega_-$ for all $t\in(0,1)$.  By the Jordan
curve theorem, it follows that the curves $\gamma_\pm$ must cross at least once in $(-T/2,T/2)\times(0,\infty)$,
yielding a contradiction.  Alternatively, observe that by joining the points $x_{1,2}$ and $y_{1,2}$
to a point $P$ in $(-T/2,T/2)\times(-\infty,0)$, and then connecting $x_1$ and $y_2$ by a curve
in $(-T/2,T/2)\times(-\infty,0)$, one embeds the complete graph $K_5$ in the plane.  Since $K_5$
is not planar, one again finds a contradiction, which establishes the desired oscillation estimate
for the even eigenfunctions.  A similar argument applies to the odd eigenfunctions. 
\end{proof}

\subsection{Proof of Nondegeneracy}

Now that we have information regarding the $T$-antiperiodic ground state eigenfunctions of $L_{\pm}$ and
the nodal patterns for their second $T$-antiperiodic eigenfunctions, we aim to establish
the nondegeneracy of the linearization $\delta^2 \mathcal{E}(\phi)$.  To this end, for each $\mu>0$ let
$\phi(\cdot;\mu)\in H^{\alpha/2}_{\rm a}(0,T)$ be a real-valued local minimizer of 
$\mathcal{E}(\cdot;\mu):=\mathcal{E}(\cdot;0,\mu)$ over $H^{\alpha/2}_{\rm a}(0,T)$ subject to fixed
$Q(u)=\mu$ and $N(u)=0$.  Then by construction, the second derivative test for constrained
extrema yields
\[
\delta^2 \CalE(\phi) \vert_{\{\delta Q(\phi),\delta N(\phi)\}^\perp} \geq 0,
\]
where here
\[
\braces{ \delta Q(\phi),\delta N(\phi) }^\perp := \braces{ h \in H^{\alpha/2}_{\rm a}([0,T];\mathbb{C}) : \innerprod{\phi}{h} =\innerprod{i\phi'}{h}= 0 } 
\]
denotes the tangent space at $\phi$ to the codimension two constrained subspace 
\[
\Sigma_\mu:=\left\{\psi\in H^{\alpha/2}(0,T):Q(\psi)=\mu,~~N(\psi)=0\right\} 
\]
in $H^{\alpha/2}_{\rm a}(0,T)$. Recall that the inner product $\langle\cdot,\cdot\rangle$ is defined throughout as
\[
\langle u,v\rangle = \Re\int_0^T u \bar{v} ~dx.
\]
By Courant's mini-max principle, this implies that the operator $\delta^2 \mathcal{E}(\phi)$
has at most two negative $T$-antiperiodic eigenvalues.  Specifically, since
$\delta Q(\phi)=\phi$ and $\delta N(\phi)=i\phi'$ are real and imaginary valued, respectively, it follows
that the linear operators $L_+$ and $L_-$ each have at most one negative $T$-antiperiodic
eigenvalue, with
\[
L_+\big{|}_{\{\delta Q(\phi)\}^\perp}\geq 0\quad\textrm{and}\quad L_-\big{|}_{\{\Im(\delta N(\phi))\}^\perp}\geq 0.
\]
A finer description of the spectral properties of $L_{\pm}$ is given below.

\begin{lemma}\label{L:specprop}
Under the hypothesis of Proposition \ref{P:nondegeneracy}, the following are true:
\begin{itemize}
\item[(i)] The operator $\delta^2 \mathcal{E}(\phi)$ acting on $L^2_{\rm a}(0,T)$
has at most one negative eigenvalue, with
\[
L_+\geq 0 \quad \textrm{and}\quad n_-(L_-) \leq  1.
\]
\item[(ii)] $\phi'\in\ker(L_+)$, and it corresponds to the ground state eigenfunction
of $L_+$ restricted to the subspace of odd functions in $H^{\alpha/2}_{\rm a}(0,T)$.
\item[(iii)] $\phi\in\ker(L_-)$, and it corresponds to the ground state eigenfunction
of $L_-$ restricted to the subspace of even functions in $H^{\alpha/2}_{\rm a}(0,T)$.
\item[(iv)] The functions $\phi'$ and $\phi^{2\sigma}\phi'$ are in the range of $L_-$.
\item[(v)] The function $\phi^{2\sigma+1}$ is in the range of $L_+$.
\end{itemize}
\end{lemma}

\begin{proof}
First, note that \eqref{E:profile_equation} is equivalent to $L_-\phi=0$, while differentiating the 
profile equation with respect to $x$ gives $L_+\phi'=0$.  Claims (ii) and (iii) now follow immediately
from the ground state theory in Theorem \ref{T:ground} and the monotonicity properties
of $\phi$ guaranteed by Lemma \ref{L:c=0_prop}.  Moreover, by the ordering of the even and odd
$T$-antiperiodic ground states of $L_{\pm}$ given by 
Proposition \ref{P:gsordering},
the fact that $\phi$ and $\phi'$ are even and odd, respectively, implies (i) by the above discussion.

Next, observing that $L_+=L_-+2\sigma\phi^{2\sigma}$, the identities $L_+\phi'=0$ and $L_-\phi=0$ 
immediately imply that $L_-\phi'=-2\sigma\phi^{2\sigma}\phi'$ and $L_+\phi=2\sigma\phi^{2\sigma+1}$.  
Finally, differentiating the profile equation \eqref{E:profile_equation} with respect to $c$ at $c=0$
gives
\[
L_+\Re\left(\frac{\partial\phi}{\partial c}\right)=-\left(\frac{\partial\omega}{\partial c}\Big{|}_{c=0}\right)\phi,
	\quad L_-\Im\left(\frac{\partial\phi}{\partial c}\right)=-\phi'.
\]
The first equation above obviously holds thanks to Lemma \ref{L:c=0_prop}, while the second
shows that $\phi'$ is in the range of $L_-$.  This establishes (iv) and (v).
\end{proof}

\begin{remark}\label{r:travel}
In the proof that $\phi'$ lies in the range of $L_-$ above, we heavily relied on the fact that
the real-valued, $T$-antiperiodic standing profile $\phi(\cdot; \mu)$ is a member of a 
more general family of complex-valued $T$-antiperiodic traveling waves $\phi(\cdot; c,\mu)$ defined for $|c|$ sufficiently small;
see Proposition \ref{P:min} and Lemma \ref{L:c=0_prop}.
In the local case $\alpha=2$, one may of course rely on the Galilean invariance of 
\eqref{E:FNLS1} to produce such a curve of traveling solutions near $c=0$, and differentiating
along this curve yields the same result.  For $\alpha\in(1,2)$, such an (exact)
Galilean invariance does not exist, which is why we had to take extra care in our existence
theory to carefully construct such a curve of solutions for a given $\mu>0$.
\end{remark}

With the above preliminaries in mind, we can now establish the main result of this section.

\

\begin{proof}[Proof of Proposition \ref{P:nondegeneracy}]
First, note that since $\phi^{2\sigma}$ is even and $T$-periodic by construction, 
the subspaces $L^2_{\rm a,odd}(0,T)$ and $L^2_{\rm a,odd}(0,T)$ of even and odd, 
respectively, $T$-antiperiodic functions are invariant subspaces of the operators $L_{\pm}$.  In particular, the operators $L_{\pm}$
respect the orthogonal decomposition
\[
L^2_{\rm a}(0,T)=L^2_{\rm a,odd}(0,T)\oplus L^2_{\rm a}(0,T)_{\rm even}
\]
so that
\[
\sigma\left(L_{\pm}\big{|}_{L^2_{\rm a}(0,T)}\right)=
\sigma\left(L_{\pm}\big{|}_{L^2_{\rm a,odd}(0,T)}\right)\cup
\sigma\left(L_{\pm}\big{|}_{L^2_{\rm a,even}(0,T)}\right).
\]
Since Lemma \ref{L:specprop} implies that	
\[
\ker\left(L_+\big{|}_{L^2_{\rm a,odd}(0,T)}\right)={\rm span}\left\{\phi'\right\}\quad\textrm{and}\quad
\ker\left(L_-\big{|}_{L^2_{\rm a,even}(0,T)}\right)={\rm span}\left\{\phi\right\},
\]
it remains to verify that $\ker\left(L_+\big{|}_{L^2_{\rm a,even}(0,T)}\right)$ and
$\ker\left(L_-\big{|}_{L^2_{\rm a,odd}(0,T)}\right)$ are trivial.

First, suppose there exists a non-trivial solution $v\in L^2_{\rm a,even}(0,T)$ of the equation $L_+v=0$.  
Since $L_+$ is self adjoint on $L^2_{\rm a}(0,T)$, 
the Fredholm alternative implies that $v$ must be orthogonal
to the range of the operator $L_+$ acting on $L^2_{\rm a}(0,T)$.  
Since zero is the ground state eigenvalue
of $L_+$ acting on $L^2_{\rm a}(0,T)$ by Lemma \ref{L:specprop}(i), 
it follows that $v$ is the even ground state eigenfunction for $L_+$ and hence,
by Theorem \ref{T:ground}, may be chosen to be strictly positive on $(-T/2,T/2)$.
To reach the desired contradiction, observe that Lemma \ref{L:specprop}(v) implies 
the function $\phi^{2\sigma+1}$ is in the range of $L_+$ acting on $L^2_{\rm a}(0,T)$.
Since $\phi$ is positive on $(-T/2,T/2)$ we have
\[
\int_0^Tv(x)\phi^{2\sigma+1}(x)dx\neq 0,
\]
contradicting the Fredholm alternative.
Consequently, $\ker(L_+|_{L^2_{\rm a,even}(0,T)})=\{0\}$ and hence
\begin{equation}\label{E:ker_Lminus}
\ker\left(L_+\right)={\rm span}\left\{\phi'\right\},
\end{equation}
verifying the nondegeneracy of $L_+$ on $L^2_{\rm a}(0,T)$.

Next, we turn our attention to $L_-$.
First, we claim that $L_-$ has exactly one negative $T$-antiperiodic eigenvalue.  To this end, suppose (to show a contradiction) that $L_-$ does not have a negative eigenvalue.  Since $L_- \phi = 0$ and $\phi$ is even, the ground state ordering 
(Proposition \ref{P:gsordering}) implies that there exists $\psi \in L_{\text{a,odd}}^2(0,T)$ such that $L_- \psi = 0$ with $\psi$ being the ground state on $L_- |_{L_{\text{a,odd}}^2(0,T)}$.  Then by Theorem \ref{T:ground} (b), $\psi$ is sign-definite on $(0,T)$, hence $\left< \psi, \phi' \right> \neq 0$, i.e. $\phi$ is not orthogonal to $\phi'$.  But $\phi' \in \range\left( L_-|_{L^2_{\text{a,odd}}(0,T)} \right) = \ker\left( L_-|_{L^2_{\text{a,odd}}(0,T)} \right)^\perp$, a contradiction of the Fredholm alternative.  Thus it must be that $n_-(L_-) \geq 1$, which implies that $n_-(L_-) = 1$ by Lemma \ref{L:specprop} (i).  Now, since $L_-\phi=0$, it follows that $\lambda=0$ is the second eigenvalue of $L_-$ acting on $L^2_{\rm a}(0,T)$.  As above, suppose there exists a nontrivial solution $v\in L^2_{\rm a,odd}(0,T)$ to the equation $L_-v=0$.  By Lemma \ref{L:oscillation}, $v$ may change signs at most twice on $(0,T)$.  We will show that such a nontrivial $v$ cannot exist by again using the Fredholm alternative.  To this end, note that if $v$ has a fixed sign on $(0,T)$, then using
that $\phi'<0$ on $(0,T)$ we have
\begin{align*}
\int_{0}^{T}v(x)\phi'(x)~dx\neq 0,
\end{align*}
which contradicts the Fredholm alternative since $\phi' \in \range(L_-)$ by Lemma \ref{L:specprop}(iv).
Thus, any non-trivial $v \in \ker\parens{L_-}$ 
must change signs at least once in $(0,T)$ and, since odd $T$-antiperiodic functions are even
about $x=T/2$, such a function must have exactly two sign changes in $(0,T)$; one at some $x=x_0\in(0,T/2)$ and the other at $x=T-x_0\in(T/2,T)$.
Define
\[
\eta(x) := \phi'(x) \parens{ \phi(x)^{2\sigma} - \phi(x_0)^{2\sigma} } 
\]
and note that $\eta\in\range(L_-)$ by Lemma \ref{L:specprop}(iv) and, further,
$\eta$ changes signs at $x=x_0$ and $x=T-x_0$, by monotonicity of $\phi$ on $(0,T)$.
Consequently, $\int_0^T\eta(x)v(x)dx\neq 0$ again contradicting the Fredholm alternative.
Thus, it must be that $\ker\parens{L_- \big|_{L^2_{\rm a,odd}(0,T)}} = \{0\}$, and we conclude that
\begin{equation}  \label{E:ker_Lplus}
\ker\parens{L_-} = {\rm span}\{ \phi \},
\end{equation}
verifying the nondegeneracy of $L_-$ on $L^2_{\rm a}(0,T)$.

Finally, since $\CalL = \diag(L_+,L_-)$ is diagonal, we have by \eqref{E:ker_Lminus} and \eqref{E:ker_Lplus} that
\[
\ker(\CalL) = {\rm span}\left\{
\left(\begin{array}{c}\phi' \\ 0 \end{array}\right),~ \left(\begin{array}{c} 0 \\ \phi \end{array}\right)\right\},
\]
establishing Proposition \ref{P:nondegeneracy}.
\end{proof}

Before proceeding, we point out an interesting observation regarding the parameterization of the profiles
$\phi(\cdot;c,\mu)$ in a neighborhood of a given $(c,\mu)=(0,\mu_0)$ with $\mu_0>0$.  From the view
point of special solutions to the given PDE, one might prefer that the profiles be parameterized the Lagrange multiplier $\omega$ instead of by the $L^2$-norm of the solution.  By the Implicit
Function Theorem, the manifold of nearby profiles $\phi(\cdot;c,\mu)$ can be reparameterized
in a $C^1$ manner by the parameters $(c,\omega)$ near $(c,\mu)=(0,\mu_0)$ provided $\omega$ depends in a $C^1$ manner on $c$ and $\mu$ and that the Jacobian matrix
\begin{equation}\label{E:jacobian1}
\frac{\partial(c,\omega)}{\partial(c,\mu)}=
\left(\begin{array}{cc}1 & 0 \\ \frac{\partial\omega}{\partial c} &\frac{\partial\omega}{\partial\mu}
\end{array}\right)
\end{equation}
is non-singular at $(c,\mu)=(0,\mu_0)$ or, equivalently, if $\frac{\partial\omega}{\partial\mu}(0,\mu_0)\neq 0$.
To see that the above Jacobian is indeed non-singular, observe that by differentiating the profile
equation \eqref{E:profile_equation} with respect to $\mu$ at $c=0$ yields\footnote{Throughout this informal discussion, we will make the blanket assumption that
$\omega$ and the profile $\phi$ depend smoothly on $c$ and $\mu$.} the identity
\[
L_+ \left( \frac{\partial\phi}{\partial\mu} \right) = -\left(\frac{\partial\omega}{\partial\mu}\right)\phi.
\]
Thus, if it were the case that $\frac{\partial\omega}{\partial\mu}=0$ at $(c,\mu)=(0,\mu_0)$,
then by Proposition \ref{P:nondegeneracy} it must be that 
$\frac{\partial\phi}{\partial\mu}=A\phi'$ for some constant $A\in\RM$, which is impossible
since $\frac{\partial\phi}{\partial\mu}$ is even, and $\phi'$ is odd.  As a result, the nondegeneracy
result Proposition \ref{P:nondegeneracy} implies that the Jacobian matrix \eqref{E:jacobian1}
is necessarily non-singular, and hence, if desired, we could consider the real-valued solutions
constructed in Lemma \ref{L:c=0_prop} to be parameterized by the wave speed $c$ and 
the temporal frequency (Lagrange multiplier) $\omega$.

However, while the parameterization by $(c,\omega)$ may seem more natural mathematically, especially
from the perspective of obtaining explicit formulae, it is actually more natural from the standpoint
of \emph{local dynamics} to attempt to reparameterize the family $\phi(\cdot; c,\mu)$ completely in terms of the conserved
quantities of the PDE flow generated by \eqref{E:FNLS1}.  Indeed, it has been observed
in several different contexts that a key ingredient in understanding the local dynamics near special solutions,
for example traveling waves, is that, locally, the manifold of special solutions can be 
parameterized by the conserved quantities of the PDE flow.  In the case of Hamiltonian evolutionary PDE,
this typically arises as a ``$\frac{dP}{dc}\neq 0"$ type condition, where the $P$ is a charge functional associated
via Noether's theorem to continuous Lie point symmetry of the equation: see \cite{PW92,BSS87,Bona75,Lin08} in the case
of traveling solitary waves, and \cite{GH07,HJ15,MJ09,BJK11,BJK14} in the case of periodic traveling waves,
while this observation has likewise been made in the context of partially dissipative systems of conservation
or balance laws in \cite{JNRZ14}.
In the present context, the profiles $\phi(\cdot;c,\mu)$ may be locally reparameterized in a $C^1$ manner 
by the conserved quantities $(N,Q)$ provided that the Jacobian matrix
\begin{equation}\label{E:jacobian2}
\frac{\partial(N,Q)}{\partial(c,\mu)}=
\left(\begin{array}{cc}\frac{\partial N}{\partial c} & \frac{\partial N}{\partial\mu} \\ 0  & 1\end{array}\right)
\end{equation}
is non-singular at a given $(c,\mu)$ or, equivalently, that $\frac{\partial N}{\partial c}\neq 0$ at the given
wave.  As seen in Theorem \ref{T:orbital_stability} below, the condition $\frac{\partial N}{\partial c}\neq 0$
ensures the nonlinear orbital stability of the waves $\phi(\cdot;c,\mu)$ constructed here.

We end this section by demonstrating that the non-singularity of the Jacobian matrix \eqref{E:jacobian2} 
ensures that the generalized $L^2_{\rm a}(0,T)$-kernel of the linearized operator associated with \eqref{E:FNLS1}
supports a Jordan structure, which plays a central role in the forthcoming stability analysis.
Note that linearizing \eqref{E:FNLSrot} about $\phi(\cdot;\mu)$ yields the linear system
\[
v_t=-i\delta^2 \mathcal{E}(\phi)v,
\]
which, by taking the Laplace transform in time and decomposing into real and imaginary parts, 
leads one to the spectral problem
\[
\left(\begin{array}{cc}0&1\\-1&0\end{array}\right)\left(\begin{array}{cc}L_+&0\\0&L_-\end{array}\right)
\left(\begin{array}{c}\Re(v)\\\Im(v)\end{array}\right)=\lambda\left(\begin{array}{c}\Re(v)\\\Im(v)\end{array}\right).
\]

\begin{proposition}[Jordan Block Structure] \label{P:Jordan}
Under the hypotheses of Proposition \ref{P:nondegeneracy}, zero is a $T$-antiperiodic generalized
eigenvalue of the linearized operator
\[
J\mathcal{L}:=\left(\begin{array}{cc}0&1\\-1&0\end{array}\right)\left(\begin{array}{cc}L_+&0\\0&L_-\end{array}\right)
\]
associated to the profile $\phi(\cdot;c=0,\mu)$ with algebraic multiplicity four and geometric multiplicity two
provided that $\frac{\partial N}{\partial c}(\phi(\cdot;c,\mu))$ is non-zero at $c=0$.
\end{proposition}

\begin{proof}
By Proposition \ref{P:nondegeneracy}, we already know that zero is a $T$-antiperiodic eigenvalue
of both $L_{\pm}$ with algebraic multiplicity at least one and geometric multiplicity precisely one.  
Since the skew-adjoint operator $J$ is invertible, it follows from Lemma \ref{L:specprop} 
that the operator $J\mathcal{L}$ has zero as a generalized eigenvalue with geometric multiplicity two and algebraic multiplicity at least four with
\[
\ker\left(J\mathcal{L}\right)={\rm span}\left\{\left(\begin{array}{c}\phi'\\0\end{array}\right),~
\left(\begin{array}{c}0\\ \phi\end{array}\right)\right\}\subset{\rm range}\left(J\mathcal{L}\right)
\]
and
\begin{equation}\label{E:genker}
J\mathcal{L}\left(\begin{array}{c}0 \\ \Im\left(\frac{\partial\phi}{\partial c}\right)\end{array}\right)
=\left(\begin{array}{c}-\phi'\\ 0 \end{array}\right),~~
J\mathcal{L}\left(\begin{array}{c}\xi\\ 0\end{array}\right)
=\left(\begin{array}{c}0\\ \phi\end{array}\right),
\end{equation}
where $\xi$ is an even,  $T$-antiperiodic solution $L_+\xi=\phi$, which is guaranteed to exist\footnote{If one assumes $\phi$ and $\omega$ depend smoothly on $\mu$,
one can take $\xi=\left(\frac{\partial\omega}{\partial\mu}\right)^{-1}\phi_\omega$, as may be verified by differentiating the profile equation
\eqref{E:profile_equation} with respect to $\mu$ at $c=0$.} by Proposition \ref{P:nondegeneracy}.
By the Fredholm alternative, the above Jordan chains terminate at height two provided that the functions
$\left(\begin{array}{c}0 \\ \Im\left(\frac{\partial\phi}{\partial c}\right)\end{array}\right)$ and
$\left(\begin{array}{c}\frac{\partial\phi}{\partial\mu}\\ 0\end{array}\right)$ are orthogonal
to $\ker(J\mathcal{L})=J\ker(\mathcal{L})$, i.e. provided that the quantities
\[
\int_0^T \phi'\Im\left( \frac{\partial \phi}{\partial c} \right)dx=-\frac{\partial N}{\partial c}(\phi)\quad\textrm{and}\quad
\int_0^T \xi\phi~dx = \left<\xi,L_+\xi\right>
\]
are both non-zero.
Since $\xi$ is even, Proposition \ref{P:nondegeneracy} implies $\left<\xi,L_+\xi\right>> 0$.  It follows that
zero has algebraic multiplicity exactly four so long as $\frac{\partial N}{\partial c}\neq 0$, as claimed.
\end{proof}

\begin{lemma}\label{L:NcSign}
If $N(\phi(\cdot;c))$ is differentiable at $c=0$, then $\frac{\partial N}{\partial c}(\phi(\cdot;c)) \Big|_{c=0} \leq 0$.
\end{lemma}

\begin{proof}
Since $\phi(\cdot;c)$ minimizes $F_c$ subject to fixed $Q$, we know that for all $c$ we have
\[
F_c(\phi(\cdot;c))\leq F_c(\overline{\phi(\cdot;c)}).
\]
It follows that
\[
\mathcal{H}(\phi(\cdot;c))+cN(\phi(\cdot;c))\leq \mathcal{H}(\overline{\phi(\cdot;c)})+cN(\overline{\phi(\cdot;c)})
\]
and hence, using the forms of $\mathcal{H}$ and $N$, that $cN(\phi(\cdot;c))\leq -cN(\phi(\cdot,c))$.  Consequently, $cN(\phi(\cdot;c))\leq 0$.
Similarly, we have that $-cN(\phi(\cdot;-c))\leq 0$ so that adding and dividing by $2c^2$ gives
\[
\frac{N(\phi(\cdot;c))-N(\phi(\cdot;-c))}{2c}\leq 0. 
\]
Taking $c\to 0$ implies that $\frac{\partial N}{\partial c}(\phi(\cdot;c)) \Big|_{c=0} \leq 0$, as claimed.
\end{proof}
Alternatively, the above lemma can also be seen from the viewpoint of an index theorem: see Remark \ref{R:sign_condition2} below.

\section{Stability of Constrained Energy Minimizers}  \label{S:stability_of_minimizers}

Let $T>0$ and $\mu_0>0$ be fixed and let $\phi_0:=\phi(\cdot;\mu_0)$ denote a real-valued, $T$-antiperiodic
solution of the nonlocal profile equation \eqref{E:profile_equation} with $c=0$ satisfying $Q(\phi_0)=\mu_0$,
whose existence is guaranteed by Proposition \ref{P:min} and Lemma \ref{L:c=0_prop}.  
The profile $\phi_0$ is thus an equilibrium solution of the PDE 
\begin{equation}\label{E:FNLSrot2}
iu_t-\omega_0u-\Lambda^\alpha u-|u|^{2\sigma}u=0,
\end{equation}
where here $\omega_0:=\omega(0,\mu_0)$.  In this section, we wish to consider the stability of $\phi_0$
under the evolution of \eqref{E:FNLSrot2} to general complex-valued, $T$-antiperiodic perturbations or, equivalently,
the stability of the standing wave solution $u(x,t;\mu_0)=e^{i\omega_0 t}\phi_0(x)$ under the evolution of
\eqref{E:FNLS1} to such perturbations.

\

For $\alpha>1$ and $\sigma>0$, an iteration argument reveals that the Cauchy problem for \eqref{E:FNLSrot2}
is locally in time well-posed in $H^{1/2+}_{\rm a}([0,T];\mathbb{C})$. Furthermore, using conservation
laws these local solutions can be extended to global ones in $H^{\alpha/2}_{\rm a}([0,T];\mathbb{C})$
provided the initial data is in $H^{\alpha/2}_{\rm a}(0,T)$.  Throughout our analysis we work on an 
appropriate subspace $X$ of $H^{\alpha/2}_{\rm a}([0,T];\mathbb{C})$ where the Cauchy problem associated with
\eqref{E:FNLSrot} is locally well-posed and where the functionals $\mathcal{H},Q,N:X\to\mathbb{R}$ are smooth. 

\

Observe that the evolution defined by \eqref{E:FNLS1}, and hence of \eqref{E:FNLSrot2}, is invariant
under a two-parameter group of symmetries generated by spatial translations and unitary
phase rotations.  For each $\phi\in X$ this motivates us to define the group orbit
\[
\mathcal{O}_\phi:=\left\{e^{i\beta}\phi(\cdot-x_0):(\beta,x_0)\in\mathbb{R}^2\right\}\subset X.
\]
Loosely speaking, we say that the standing wave $\phi_0(\cdot;\mu_0)$ is \emph{orbitally stable} 
if the group orbit $\mathcal{O}_{\phi_0}$ is stable under the evolution of \eqref{E:FNLSrot2}, i.e. 
if solutions of \eqref{E:FNLSrot2} remain close in the $X$-norm to $\mathcal{O}_{\phi_0}$ 
for all future times provided their initial data is sufficiently close in the $X$-norm 
to $\mathcal{O}_{\phi_0}$.  We elaborate further below.

\

Our treatment of the orbital stability problem is inspired by the Lyapunov method.  Setting
\begin{equation}\label{E:lag2}
\mathcal{E}_0(u) :=\mathcal{H}(u)+\omega_0 Q(u),
\end{equation}
we recall from Lemma \ref{L:c=0_prop} that $\phi_0$
is a critical point of $\mathcal{E}_0$, i.e. that $\delta\mathcal{E}_0(\phi)=0$, or, equivalently,
that $\phi_0$ is a critical point of $\mathcal{H}$ subject to fixed $Q(u)=\mu_0$ and $N(u)=0$.
Furthermore, Proposition \ref{P:nondegeneracy} implies that the kernel of 
the Hessian $\delta^2 \mathcal{E}_0(\phi_0)$ is generated by the translation and phase rotation symmetries.  Intuitively, we expect the group orbit of $\phi_0$ to be stable provided the operator 
$\delta^2 \mathcal{E}_0(\phi_0)$ is convex (in an appropriate sense) at $\phi_0$.
Below, we shall demonstrate this convexity and hence establish the orbital stability
of the standing wave $\phi_0$ under the evolution of \eqref{E:FNLSrot2}.  
We now state the main result of this section.

\begin{theorem}[Orbital Stability]  \label{T:orbital_stability}
Suppose $\alpha \in (1,2]$, and let $\phi_0 = \phi(\cdot;\mu_0)=\phi(\cdot;c=0,\mu_0)\in H^{\alpha/2}_{\rm a}(0,T)$ 
be a real-valued, 
$T$-antiperiodic local minimizer of $\CalH$ subject to $Q(u)=\mu_0$ and $N(u)=0$ as constructed
in Lemma \ref{L:c=0_prop}.  Suppose in addition that both $\phi$ and the associated Lagrange multiplier $\omega$
depend on $\mu$ and $c$ in a $C^1$ manner near $(\mu,c)=(\mu_0,0)$.
If \begin{equation}\label{E:pos}
\partial_c N(\phi(\cdot;c,\mu_0))<0\quad\textrm{at}\quad c=0,
\end{equation}
then for all $\varepsilon > 0$ sufficiently small, there exists a constant $C = C(\varepsilon)$ such that 
if $v \in X$ with $\norm{v}_X \leq \varepsilon$ and $N(\phi(\cdot;\mu_0)+v)=0$, and if 
$u(\cdot,t)$ is a local in time solution  of \eqref{E:FNLSrot2} with initial data $u(\cdot,0) = \phi_0 + v$, 
then $u(\cdot,t)$ can be continued to a solution for all $t>0$ and
\[
\sup_{t>0}\inf_{(\beta,x_0)\in\mathbb{R}^2}\left\|u(\cdot,t)-e^{i\beta}\phi_0(\cdot-x_0)\right\|_{X} \leq C\|v\|_{X}.
\]
\end{theorem}

\begin{remark}\label{r:mu_stability} 
As with the nondegeneracy theory in Section \ref{S:nondegeneracy} above, the smooth dependence of $\phi$ and $\omega$ on the wave speed
$c$ near $c=0$ is crucial to our stability analysis.  The requirement that these quantities \emph{also} depend smoothly on the charge $\mu$
is needed to allow perturbations $v\in X$ for which $Q(\phi(\cdot;\mu_0)+v)\neq\mu_0$, i.e. initial perturbations that slightly change the charge from the underlying
wave.  If one is willing to restrict to the class of perturbations for which  $Q(\phi(\cdot;\mu_0)+v)=\mu_0$, thereby ensuring the underlying and perturbed waves have the same
charge, then the smooth dependence of $\phi$ and $\omega$ is not needed; see the proof of Theorem \ref{T:orbital_stability} below.
Note that since we allow for $Q(\phi(\cdot;\mu_0)+v)\neq\mu_0$, the forthcoming stability theory is developed ``by hand" instead of simply applying 
the results of \cite{GSS1,GSS2} directly.
\end{remark}

\begin{remark}\label{r:sign_condition}
Due to Lemma \ref{L:NcSign}, we need only assume that $\partial_c N(\phi(\cdot;c,\mu_0))\neq 0$ in order to satisfy the sign condition \eqref{E:pos}.  Indeed, it is the nonvanishing of $\partial_c N(\phi(\cdot;c,\mu_0))$ that will be used in the forthcoming proof.
\end{remark}

To begin the proof of Theorem \ref{T:orbital_stability}, observe Proposition \ref{P:nondegeneracy} and Lemma
\ref{L:specprop} imply that $\phi_0$ is a degenerate saddle point of $\delta^2 \mathcal{E}_0$ acting
on $H^{\alpha/2}_{\rm a}(0,T)$, with one negative direction and two neutral directions.  To handle these potentially
unstable directions, we note that the evolution of \eqref{E:FNLSrot2} does not occur on the whole space
$X$, but rather on the codimension two nonlinear manifold
\[
\Sigma_0:=\left\{u\in X:Q(u)=\mu_0,~~N(u)=0\right\}.
\]
In particular, $\Sigma_0$ is an invariant set under the flow of \eqref{E:FNLSrot2},
with $\mathcal{O}_{\phi_0}\subset\Sigma_0$.  The key step in the proof of Theorem \ref{T:orbital_stability},
we establish the coercivity of $\mathcal{E}_0$ on $\Sigma_0$ in a neighborhood of $\mathcal{O}_{\phi_0}$
provided that condition \eqref{E:pos} holds.  To this end, we define
\begin{equation}\label{E:tanspace}
\mathcal{T}_0:={\rm span}\left\{\delta Q(\phi_0),\delta N(\phi_0)\right\}^\perp
	={\rm span}\{\phi_0,i\phi_0'\}^\perp
\end{equation}
to be the tangent space in $X$ to the submanifold $\Sigma_0$ at $\phi_0$, and establish
the following technical result.

\begin{lemma}  \label{L:infimum}
Under the hypothesis of Theorem \ref{T:orbital_stability},
\[
\inf\left\{ \innerprod{\delta^2 \CalE_0(\phi_0)v}{v} : \norm{v}_X=1,~ v\in\mathcal{T}_0,~ \ v \perp {\rm span}\{\phi_0',i\phi_0\} \right\} > 0.
\]
In particular, there exists a constant $C>0$ such that
\[
\left<\delta^2 \mathcal{E}_0(\phi)v,v\right>\geq C\|v\|_{X}^2
\]
for all $v\in \mathcal{T}_0$ with $v\perp{\rm span}\{\phi_0',i\phi_0\}$.
\end{lemma}

\begin{proof}
First, let $\Pi:L^2_{\rm a}(0,T)\to\mathcal{T}_0$ be the self-adjoint projection onto $\mathcal{T}_0$
and note consider the constrained operator
\[
\Pi\delta^2 \mathcal{E}_0(\phi_0):\mathcal{T}_0\subset X\to\mathcal{T}_0.
\]
By construction, $\mathcal{T}_0$ is an invariant subspace of $\Pi\delta^2 \mathcal{E}_0(\phi_0)$ 
and, further, the operator $\Pi\delta^2 \mathcal{E}_0(\phi_0)$ is self adjoint acting on $\mathcal{T}_0$.
Owing to the periodic boundary conditions, the spectrum of the operator 
$\Pi\delta^2 \mathcal{E}_0(\phi_0)$ is real and purely discrete, consisting of infinitely many
isolated eigenvalues with no finite accumulation point.  Further, since $\phi_0$ locally
minimizes $\mathcal{E}_0$ with fixed $Q(u)=\mu_0$ and $N(u)=0$, we have that
\[
\Pi\delta^2 \mathcal{E}_0(\phi_0)\geq 0.
\]
It follows that $\Pi\delta^2 \mathcal{E}(\phi_0)$ is positive-semidefinite acting on $\mathcal{T}_0$,
and hence the eigenvalues of $\Pi\delta^2 \mathcal{E}_0(\phi_0)$ may be listed, not counting multiplicity, as
\[
\Lambda_1=0<\Lambda_2<\Lambda_3<\Lambda_4<\ldots\to+\infty.
\]
Thanks to the spectral gap between $\Lambda_1$ and $\Lambda_2$, if we now define
the self-adjoint spectral projection
\[
\Pi_0:\mathcal{T}_0\mapsto{\rm ker}\left(\Pi\delta^2 \mathcal{E}_0(\phi_0)\right)\cap\mathcal{T}_0
\]
it follows that
\[
\sigma\left((1-\Pi_0)\Pi\delta^2 \mathcal{E}_0(\phi_0)\right)
 		=\sigma\left(\Pi\delta^2 \mathcal{E}_0(\phi_0)\right)\setminus\{0\}
\]
so that, in particular,
\[
\inf\left\{ \innerprod{\delta^2 \CalE(\phi)v}{v} : \norm{v}_X=1,~ v\in\mathcal{T}_0, \ 
	v \perp {\rm ker}\left(\Pi\delta^2 \mathcal{E}_0(\phi_0)\right) \right\} =\Lambda_2>0.
\]
This immediately provides the estimate
\[
\left<\Pi\delta^2 \mathcal{E}_0(\phi_0)v,v\right>\geq\Lambda_2\|v\|_{L^2(0,T)}^2
\]
for all $v\in\mathcal{T}_0$ with $v\perp\ker\left(\Pi\delta^2 \mathcal{E}_0(\phi_0)\right)$.  By the definition
of the space $X$, it is now easy to see this in turn implies the estimate
\[
\left<\Pi\delta^2 \mathcal{E}_0(\phi_0)v,v\right>\geq\Lambda_2\|v\|_{X}^2
\]
for all such $v$; see \cite[Lemma 5.2.3]{KP13}.

By Proposition \ref{P:nondegeneracy}, the desired result follows immediately provided
\begin{equation}\label{E:equivker}
\ker\left(\Pi\delta^2 \mathcal{E}_0(\phi_0)\right)=\ker\left(\delta^2 \mathcal{E}_0(\phi_0)\right)
={\rm span}\{\phi_0',i\phi_0\}.
\end{equation}
Well, if $\psi\in\mathcal{T}_0$ lies in the kernel of $\Pi\delta^2 \mathcal{E}_0(\phi_0)$ then, 
by the definition of the projection $\Pi$, we have
\[
\delta^2 \mathcal{E}_0(\phi_0)\psi=A\phi_0+iB\phi_0'
\]
for some constants $A,B\in\RM$.  From Proposition \ref{P:nondegeneracy} and \eqref{E:genker},
it follows that all solutions of the above equation are of the form
\[
\psi=-A\left(\frac{\partial\omega}{\partial\mu}\right)^{-1}\frac{\partial\phi_0}{\partial\mu}
-iB\Im\left(\frac{\partial\phi_0}{\partial c}\right)+\gamma_1\phi_0'+i\gamma_2\phi_0
\]
for some constants $\gamma_1,\gamma_2\in\RM$.  The requirement that $\psi\in\mathcal{T}_0$
now enforces
\[
\left<\phi_0,\psi\right>=-A\left(\frac{\partial\omega}{\partial\mu}\right)^{-1}
	\int_0^T\frac{\partial\phi_0}{\partial\mu}\phi_0~dx=0
\]
and
\[
\left<i\phi_0',\psi\right>=-B\int_0^T\phi_0'\Im\left(\frac{\partial\phi_0}{\partial c}\right)dx =0.
\]
Since $\partial_\mu Q(\phi_0(\cdot;c,\mu))=1$ for all $\mu>0$, and 
$\partial_cN(\phi_0(\cdot;c,\mu))\neq 0$ at $(c,\mu)=(0,\mu_0)$ 
by assumption, it follows that while the operator $\Pi\delta^2 \mathcal{E}_0(\phi_0)$ formally sends the functions
$\frac{\partial\phi_0}{\partial\mu}$ and $\frac{\partial\phi_0}{\partial c}$ to zero, these functions
do not lie in the admissible space space $\mathcal{T}_0$ under the given hypotheses.
In particular, it follows that $A=B=0$ above, and hence that \eqref{E:equivker} holds, completing the proof.
\end{proof}

\

\begin{remark}\label{R:rem1}
Alternatively to the direct proof above, one may use an index formula to verify the constrained kernel
condition \eqref{E:equivker}.  Indeed, observing that since $\phi_0,i\phi_0'\in\ker\left(\delta^2 \mathcal{E}_0(\phi_0)\right)^\perp$ and that 
$\left(\delta^2 \mathcal{E}_0(\phi_0)\right)^{-1}\phi_0
	=-\left(\frac{\partial\omega}{\partial\mu}\right)^{-1}\frac{\partial\phi_0}{\partial\mu}$,
and $\left(\delta^2 \mathcal{E}_0(\phi_0)\right)^{-1}(i\phi_0')=-\frac{\partial\phi_0}{\partial c}$
it follows that
\begin{equation}\label{E:index}
\dim\left(\ker\left(\Pi\delta^2 \mathcal{E}_0(\phi_0)\right)\right)=
\dim\left(\ker\left(\delta^2 \mathcal{E}_0(\phi_0)\right)\right)+
z\left(\frac{\partial(N,Q)}{\partial(c,\mu)}\Big{|}_{(c,\mu)=(0,\mu_0)}\right)
\end{equation}
where $z(D)$ denotes the number of zero eigenvalues of a given matrix $D$;
see, for instance, \cite[Theorem 2.1]{KP12} or \cite[Theorem 5.3.2]{KP13}.  The above Jacobian has already been computed
in \eqref{E:jacobian2} and shown to be non-singular under the condition that $\frac{\partial N}{\partial c}\neq 0$
at $\phi_0$, yielding the equivalence of the kernels in \eqref{E:equivker}
\end{remark}

\

\begin{remark}\label{R:sign_condition2}
The effect of the sign condition \eqref{E:pos} can be understood through the identity
\[
n_-\left(\ker\left(\Pi\delta^2 \mathcal{E}_0(\phi_0)\right)\right)=
n_-\left(\ker\left(\delta^2 \mathcal{E}_0(\phi_0)\right)\right)-
n_-\left(\frac{\partial(N,Q)}{\partial(c,\mu)}\Big{|}_{(c,\mu)=(0,\mu_0)}\right)-z\left(\frac{\partial(N,Q)}{\partial(c,\mu)}\Big{|}_{(c,\mu)=(0,\mu_0)}\right),
\]
which again follows by the index theorem as in \cite[Theorem 5.3.2]{KP13}.  Since $\phi$ minimizes $\mathcal{E}_0$ with fixed $Q(u)=\mu_0$ and $N(u)=0$, we know
that $n_-\left(\ker\left(\Pi\delta^2 \mathcal{E}_0(\phi_0)\right)\right)=0$.  Similarly, we have by Proposition \ref{P:nondegeneracy} that
$n_-\left(\ker\left(\delta^2 \mathcal{E}_0(\phi_0)\right)\right)=1$, hence the above result implies that
$\frac{\partial(N,Q)}{\partial(c,\mu)}\Big{|}_{(c,\mu)=(0,\mu_0)}$ is either singular or has one negative eigenvalue.  This reproduces the sign condition \eqref{E:pos} by
the identity \eqref{E:jacobian2}.
\end{remark}

\

\begin{remark}\label{R:omega_not_needed1}
The smooth dependence of $\omega$ and $\phi$ on $\mu$ may be removed as an assumption from Lemma \ref{L:infimum}.  Indeed,
defining $\xi$ to satisfy $L_+\xi=\phi_0$ as in the proof of Proposition \ref{P:Jordan}, the function $\psi$ in the above proof
can be taken to be
\[
\psi=A\xi-iB{\rm Im}\left(\frac{\partial\phi_0}{\partial c}\right)+\gamma_1\phi_0'+i\gamma_2\phi_0
\]
for $\gamma_1,\gamma_2\in\RM$.   Since we have
\[
\left<\phi_0,\psi\right>=\left<L_+\xi,\xi\right>\neq 0
\]
from the proof of Proposition \ref{P:Jordan}, the requirement that $\psi\in\mathcal{T}_0$ then becomes $\partial_c\big{|}_{c=0}N(\phi_0(\cdot;c,\mu))\neq 0$, as claimed.
Note, however, that the smooth dependence of $\omega$ as a function of $\mu$ is necessary to prove Theorem \ref{T:focusing_orbital_stability} as stated, but again may be removed
as an assumption if one further restricts the class of perturbations; see Remark \ref{R:omega_not_needed2} below.
\end{remark}
\

Next, we introduce the semidistance $\rho$ on $X$ defined via
\[
\rho(u,v):=\inf_{(\beta,x_0)\in\mathbb{R}^2}\left\|u-e^{i\beta}v(\cdot-x_0)\right\|_{X},
\]
and observe that $\rho(u,v)$ simply measures the distance in $X$ from $u$ 
to the group orbit $\mathcal{O}_v$ or, equivalently, from $v$ to $\mathcal{O}_u$.
Next, we show that the functional $\mathcal{E}_0$ is coercive on the nonlinear manifold
$\Sigma_0$ with respect to the semidistance $\rho$.

\begin{proposition}[Coercivity]  \label{P:coercivity}
Under the hypothesis of Theorem \ref{T:orbital_stability}, there exist constants $\eps>0$
and $C=C(\eps)>0$ such that if $u\in\Sigma_0$ with $\rho(u,\phi_0)<\eps$, then
\[
\mathcal{E}_0(u)-\mathcal{E}_0(\phi_0)\geq C\rho(u,\phi_0)^2.
\]
\end{proposition}

\begin{proof}
By the implicit function theorem, for $\varepsilon > 0$ sufficiently small there exists a 
neighborhood $\CalU_\varepsilon := \{ u \in X : \rho(u,\phi_0) < \varepsilon \}$ of 
$\CalO_{\phi_0}$ and continuously differentiable  maps $\tau,\beta : \CalU_\varepsilon \to \R$ such that
\begin{equation}\label{E:implicit}
\tau(\phi_0)=0, \quad \beta(\phi_0)=0, \quad \innerprod{e^{i\beta(u)} u(\cdot + \tau(u))}{\phi_0'} = 0, \quad \text{and} \quad \innerprod{e^{i\beta(u)} u(\cdot + \tau(u))}{i\phi_0} = 0
\end{equation}
for all $u \in \CalU_\varepsilon$.  Since $\CalE_0$ is invariant under spatial translations, it will suffice to show that $\CalE_0(u(\cdot + \tau(u)) - \CalE_0(\phi_0) \geq C\rho(u(\cdot + \tau(u)),\phi_0)^2$.  Now, fix $u \in \CalU_\varepsilon \cap \mathcal{T}_0$ and note we can write
\begin{align*}
e^{i\beta(u)}u(\cdot + \tau(u)) &= \phi_0 + C_1\phi_0 + C_2 \phi_0' + y,
\end{align*}
where the $C_1,C_2\in\CM$ and $y\in\mathcal{T}_0$ with $y\perp\ker\left(\delta^2 \mathcal{E}(\phi_0)\right)$.
In particular, note that $C_1=C_2=y=0$ at $u=\phi_0$.  Furthermore, under the above decomposition the 
the orthogonality conditions in \eqref{E:implicit} reduce to
\[
\Re(C_2)\int_0^T(\phi_0')^2 \, dx = \Im(C_1)\int_0^T\phi_0^2 \, dx=0,
\]
and hence $C_1\in\RM$ and $C_2=iC_3$ for some $C_3\in\RM$.

Next, set
\[
h:= e^{i\beta(u)}u(\cdot+\tau(u))-\phi=C_1\phi_0+iC_3\phi_0'+y
\]
and note, after possibly translating and rotating $\phi_0$, we may assume without loss of generality 
that $\|h\|_X<\eps$.  Since the functionals $Q$ and $N$ are 
left invariant by spatial translations and unitary phase rotations, Taylor's theorem yields
\[\left\{\begin{aligned}
Q(u)&=Q(e^{i\beta(u)} u(\cdot + \tau(u))) = Q(\phi_0) + \innerprod{\delta Q(\phi_0)}{h} + \mathcal{O}(\norm{h}_X^2) \\
N(u)&=N(e^{i\beta(u)} u(\cdot + \tau(u))) = N(\phi_0) + \innerprod{\delta N(\phi_0)}{h} + \mathcal{O}(\norm{h}_X^2)
\end{aligned}\right.
\]
so that, since $u\in\Sigma_0$,
\begin{align*}
C_1\|\phi_0\|_X^2=\innerprod{\delta Q(\phi_0)}{h}=\mathcal{O}(\|h\|_X^2)\quad\textrm{and}\quad
C_3\|\phi_0'\|_X^2=\innerprod{\delta N(\phi_0)}{h}=\mathcal{O}(\|h\|_X^2).
\end{align*}
Consequently, $C_1,C_3=\mathcal{O}\left(\|h\|_{X}^2\right)$.

Finally, using Taylor's theorem again, we find
\begin{align*}
\mathcal{E}_0(u)&=\mathcal{E}_0(e^{i\beta(u)} u(\cdot + \tau(u))) = \CalE_0(\phi_0) + \innerprod{\delta \CalE_0(\phi_0)}{h} + \frac{1}{2}\innerprod{\delta^2 \CalE_0(\phi_0)h}{h} + o(\norm{h}_X^2)
\end{align*}
so that
\begin{align*}
\CalE_0(u)-\CalE_0(\phi_0) &= \frac{1}{2}\innerprod{\delta^2 \CalE_0(\phi_0)h}{h} + o(\norm{h}_X^2)\\
&= \frac{1}{2}\left<\delta^2 \mathcal{E}_0(\phi_0)y,y\right> + \mathcal{O}(C_1^2 + C_3^2) + \mathcal{O}((|C_1|+|C_3|)\|h\|_X) + o(\|h\|_X^2) \\
&= \frac{1}{2}\left<\delta^2 \mathcal{E}_0(\phi_0)y, y\right> + o(\|h\|_X^2).
\end{align*}
where the last equality is justified by the above estimates on $C_1,C_3$.
Since $y\in\mathcal{T}_0$ with $y\perp{\rm span}\{\phi_0',i\phi_0\}$, it follows by Lemma \ref{L:infimum} that
\[
\left<\delta^2 \mathcal{E}_0(\phi_0)y,y\right>\geq C\|y\|_{X}^2
\]
for some constant $C>0$.  Since the estimates $C_1,C_3=\mathcal{O}(\|h\|_{X}^2)$ yield
\begin{align*}
\norm{y}_X &= \norm{h - C_1\phi - iC_3\phi}_X 
\geq \norm{h}_X - C^* \norm{h}_X^2
\end{align*}
for some constant $C^*>0$, it follows from the definition of $h$ that for $\eps>0$ sufficiently small we have 
\[
\mathcal{E}_0(u)-\mathcal{E}_0(\phi)\geq C\left\|e^{i\beta(u)}u(\cdot+\tau(u))-\phi\right\|_{X}^2
\geq C\rho(u,\phi)^2,
\]
for some constant $C>0$, which completes the proof.
\end{proof}


Equipped with the coercivity estimate in Proposition \ref{P:coercivity} above, 
we now establish orbital stability of $\phi_0$ with respect to complex-valued, $T$-antiperiodic perturbations.

\begin{proof}[Proof of Theorem \ref{T:orbital_stability}]
Let $\eps_0>0$ be such that Proposition \ref{P:coercivity} holds, and let $v \in X$ satisfy $\rho(\phi_0,\phi_0+v) \leq \eps$ 
for some $0<\eps<\eps_0$.
Since $\phi_0$ is a critical point of $\CalE_0$, Taylor's theorem implies that
$\mathcal{E}_0(\phi_0+v)-\mathcal{E}_0(\phi_0)\leq C\eps^2$ for some constant $C>0$.
Furthermore, notice that if $\phi_0+v \in \Sigma_0$, then the unique solution $u(\cdot,t)$ of \eqref{E:FNLSrot2}
with $u(\cdot,0)=\phi_0+v$ remains in $\Sigma_0$ so long as it exists.  
Since $\CalE_0(u(\cdot,t)) = \CalE_0(u(\cdot,0)) = \CalE_0(\phi+v)$ independently of $t$, 
we have by the coercivity estimate in Proposition \ref{P:coercivity} 
that
\[
C^{-1}\rho(u(\cdot,t),\phi_0)^2\leq\mathcal{E}_0(\phi_0+v)-\mathcal{E}_0(\phi_0)\leq C\eps^2
\]
for some constant $C>0$ and all $t\geq 0$, establishing the orbital stability of $\phi_0$ to such perturbations.

In the case that $\phi_0+v \notin \Sigma_0$ but $\|v\|_X\leq\eps$ and $N(\phi_0+v)=0$, 
we utilize the nondegeneracy of the constraint set
to establish stability.  Specifically, recall that from the discussion following the proof
of Proposition \ref{P:nondegeneracy} that the  condition \eqref{E:pos} implies that the mapping 
\[
(c,\mu) \mapsto \left(N(\phi(\cdot;c,\mu)),Q(\phi(\cdot;c,\mu))\right)
\]
is a period-preserving diffeomorphism from a neighborhood of $(c,\mu)=(c_0,\mu_0)$ onto a neighborhood of 
$(N,Q)=(0,\mu_0)$.  
We may thus find a number $\tilde\mu\in\mathbb{R}$ with $\tilde\mu=\mathcal{O}(\eps)$ such that 
$\phi_\eps(\cdot;\mu_0+\tilde\mu)$ is a real-valued $T$-antiperiodic standing wave of fNLS \eqref{E:FNLS1} satisfying 
$Q(\phi_\eps(\cdot;\mu_0+\tilde\mu)) = Q(\phi_0+v)$.  
Defining $\CalE_\varepsilon(u) = \CalE(u) + \omega(\mu+\tilde\mu) Q(u)$, 
we may furthermore assume that $\phi_\varepsilon$ 
minimizes $\CalE_\varepsilon$ subject to the constraint that $Q(u)=Q(\phi_0+v)$ and $N(u)=0$.
From the proof of Proposition \ref{P:coercivity}, we now have
\[
\CalE_\varepsilon(u) - \CalE_\varepsilon(\phi_\varepsilon) \geq C\rho(u,\phi_\varepsilon)^2
\]
so long as $\rho(u,\phi_\varepsilon)$ sufficiently small and, since $\phi_\varepsilon$ is a critical point of 
$\CalE_\varepsilon$, we also have by the triangle inequality that
\[
\CalE_\varepsilon(u(\cdot,t)) - \CalE_\eps(\phi_\varepsilon) = \CalE_\varepsilon(\phi_0+v) 
- \CalE_\varepsilon(\phi_\varepsilon) \leq C\varepsilon^2
\]
for all $t \geq 0$.  Again using the triangle inequality, we finally have
\begin{align*}
\rho(u(\cdot,t),\phi_0)^2 &\leq C\parens{ \rho(u(\cdot,t),\phi_\eps)^2 + \rho(\phi_\eps,\phi_0)^2 } \\
&\leq C\parens{ \CalE_\varepsilon(u) - \CalE_\varepsilon(\phi_\varepsilon) } + C\norm{\phi_\varepsilon - \phi_0}_X^2 
= 2C\varepsilon^2 
\end{align*}
for all $t\geq 0$, implying that $\phi_0$ is orbitally stable to small perturbations
that ``slightly" change $Q$ yet preserve $N$, thus establishing Theorem \ref{T:orbital_stability}.
\end{proof}

\begin{remark}\label{R:omega_not_needed2}
From the above proof, it is clear that the assumption that $\phi$ and $\omega$ depend smoothly on $\mu$ in Theorem \ref{T:orbital_stability} is only necessary in considering
perturbations $v$ above such that $\phi_0+v\notin\Sigma_0$ with $\|v\|_X\leq\eps$ and $N(\phi_0+v)=0$; see also Remark \ref{R:omega_not_needed1} above.  Consequently, if one
is willing to restrict to perturbations $v$ that do \emph{not change} the charge and angular momentum of the underlying wave,
this assumption may be removed.
\end{remark}

\section{Analysis of the Focusing Case}  \label{S:focusing_case}

As described in the introduction, following Gallay \& Haragus \cite{GH07} 
we choose not to give full attention to periodic 
standing waves in the focusing case ($\gamma =+1$).  In large part, this choice is due to the 
observation of Rowlands \cite{Rowlands74} that, at least in the local case $\alpha=2$ with $\sigma=1$, 
all such waves are modulationally unstable; see also the recent works \cite{DS16,GCT}.  
While it is not immediately clear that Rowlands' result extends 
to the nonlocal case $\alpha\in(0,2)$, it seems reasonable to expect. 
Thus, while we may be able to establish the stability of antiperiodic waves in this case to a restricted
class of perturbations, these waves are expected to be unstable to more general periodic perturbations.
Nevertheless, the theory developed in the previous sections is still able to establish the nondegeneracy
of all the waves we construct here, irrespective of whether such waves are constrained energy minimizers,
an observation we believe is worth discussing in some detail.

To motivate the existence of such waves, observe that in the local case $\alpha=2$ with $\gamma=+1$ 
elementary phase plane analysis reveals that the focusing profile equation
\begin{equation} \label{E:focusing_profile_equation}
\Lambda^\alpha \phi +\omega\phi + ic\phi' - |\phi|^{2\sigma}\phi = 0,~~\omega,c\in\mathbb{R}
\end{equation}
for $c=0$ yields several families of real-valued, bounded periodic solutions.  Indeed, integrating \eqref{E:focusing_profile_equation} implies such waves can be
reduced to quadrature via 
\[
\frac{1}{2}\left(\phi'\right)^2=H-V(\phi;\omega),
\]
where
\[
V(\phi;\omega) = -\frac{\omega}{2}\phi^2 + \frac{1}{2\sigma+2}\phi^{2\sigma+2};
\]
see Figure \ref{F:focusing_potential}. When $\omega>0$, there exist sign-definite periodic solutions 
as well as sign changing antiperiodic solutions, with these classes of solutions being separated in 
phase space by a separatrix corresponding to the unique (up to symmetries)
solitary standing wave: mark that no such homoclinic structures exist in the defocusing case.
Furthermore, when $\omega<0$ all periodic solutions are in fact antiperiodic and there exist no other 
real-valued, nonconstant standing structures.  The nonlinear stability of such periodic structures
in the local case $\alpha=2$ has been studied in \cite{Pava07,GH07,GCT}.  As in the previous sections,
our goal here is to extend this local analysis to the genuinely nonlocal case.

\begin{figure}[t]
\centering
\begin{subfigure}[t]{0.5\textwidth}
\includegraphics[scale=0.12]{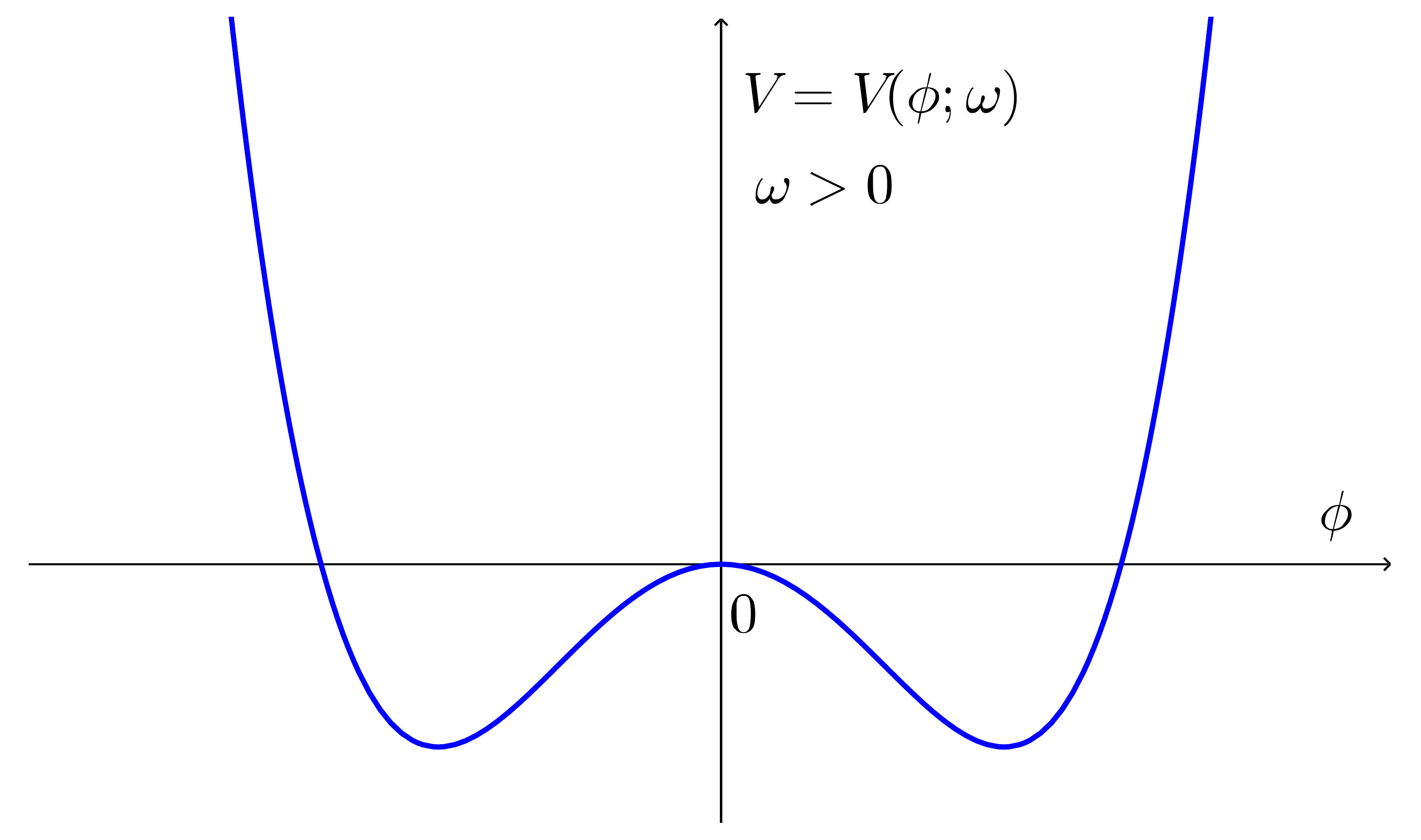}
\caption{Focusing potential for $\omega > 0$.}
\end{subfigure}%
\begin{subfigure}[t]{0.5\textwidth}
\includegraphics[scale=0.12]{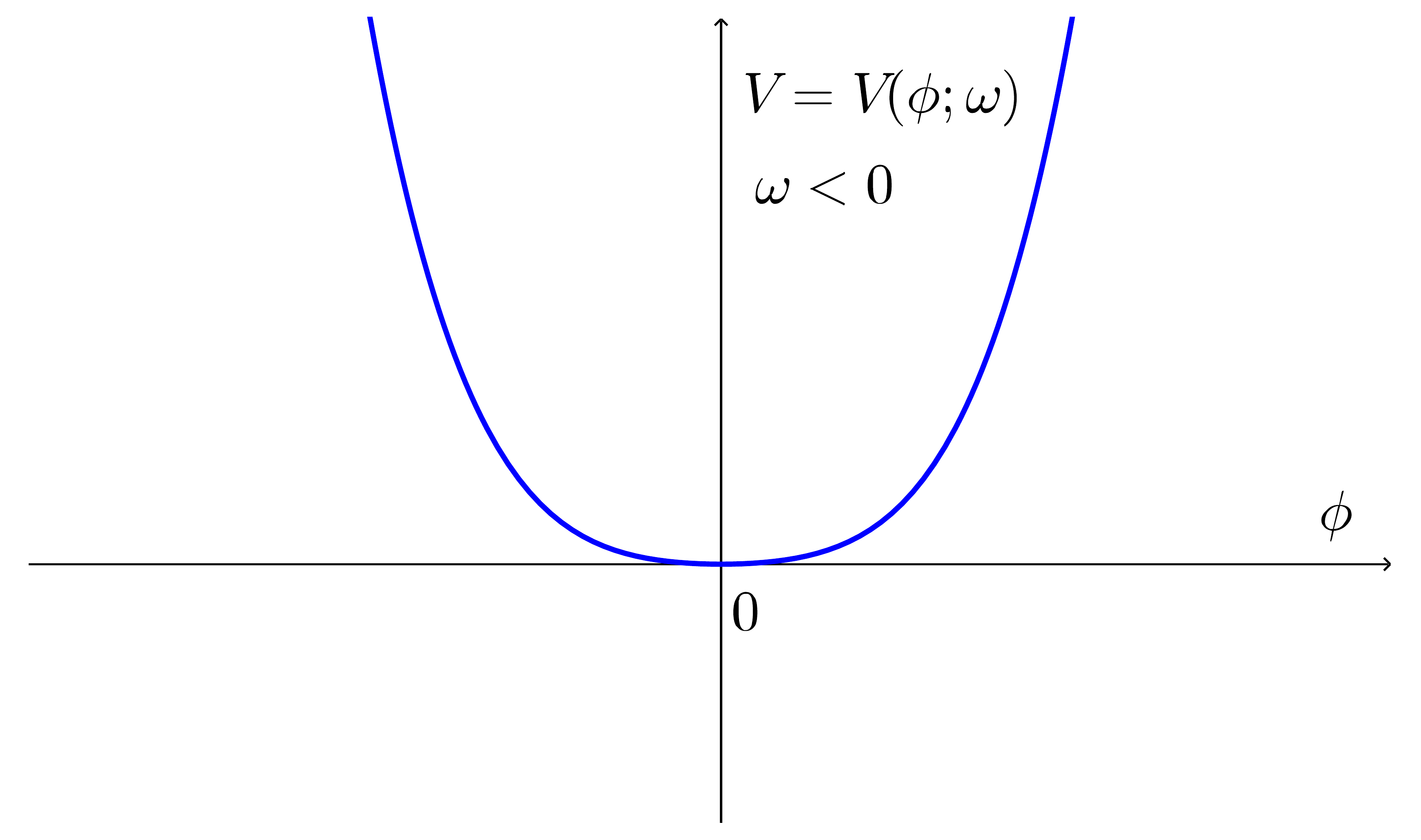} 
\caption{Focusing potential for $\omega < 0$.}
\end{subfigure}
\caption{The effective potential $V(\phi;\omega) = -\frac{\omega}{2}\phi^2 + \frac{1}{2\sigma+2}\phi^{2\sigma+2}$ for the focusing NLS.}
\label{F:focusing_potential}
\end{figure}

When $\alpha\in(1,2)$, it is natural to expect the existence of both antiperiodic and sign-definite waves
for the focusing fNLS.  Mimicking the analysis from the defocusing case, we concentrate here on the existence,
nondegeneracy, and stability of antiperiodic standing wave solutions: see Remark \ref{R:focusing_signdefinite}
for discussion regarding the sign-definite solutions. 
The existence of real-valued, antiperiodic solutions
of \eqref{E:focusing_profile_equation} may be established 
by a similar method as in Section \ref{S:existence}, though with a different
functional setup. Observe that $T$-antiperiodic solutions of \eqref{E:focusing_profile_equation} are critical points
of the Lagrangian functional
\[
H^{\alpha/2}_{\rm a}(0,T)\ni u\mapsto K(u)-P(u)+\omega Q(u),
\]
where $K$, $P$, and $Q$ are defined as in Section \ref{S:existence}: note the sign difference 
on the potential energy term $P$ between here and the defocusing case.  As the Hamiltonian
$K-P$ is not sign-definite in this case, rather than constructing 
critical points by minimizing $K-P$ subject to fixed $Q$, producing constrained energy minimizers,
we seek to minimize $K+\omega Q$ for a given $\omega$ subject to fixed potential
energy.

\begin{proposition}\label{P:focusing_existence}
Let $\alpha \in (1,2)$, $\omega\in\RM$ and $T,\sigma>0$ be fixed in the focusing 
($\gamma=+1$) fNLS \eqref{E:FNLS1}.  For each $|\omega|<\left(\frac{\pi}{T}\right)^{\alpha}$ and
$P_0>0$ there exists a real-valued, even $\phi\in H^{\alpha/2}_{\rm a}(0,T)$ with $P(\phi)=P_0$
such that $\phi$ is strictly 
decreasing on $(0,T)$ and solves \eqref{E:focusing_profile_equation} with $c=0$ in the sense
of distributions.  The function $\phi(\cdot;\omega,P_0)$
belongs to $H^\infty_{\rm a}(0,T)$ and
minimizes the Lagrangian functional
\[
\mathcal{E}(u;\omega):=K(u)-P(u)+\omega Q(u)
\]
subject to the constraint $P(u)\equiv P_0$.
\end{proposition}

To prove the above proposition, the following Poincar\'e inequality will be useful.

\begin{lemma}[Poincar\'e] \label{L:poincare}
For all $\alpha>0$, we have $Q(u) \leq \left( \frac{T}{\pi} \right)^\alpha K(u)$ for all 
$u \in H^{\alpha/2}_{\rm a}(0,T)$.
\end{lemma}

\begin{proof}
By Parseval, we have
\[
Q(u) = \frac{1}{2} \| \hat{u} \|^2_{\ell^2} \leq \frac{1}{2} \left( \frac{T}{\pi} \right)^\alpha \left\| \left| \frac{\pi n}{T} \right|^{\alpha/2} \hat{u}(n) \right\|_{\ell^2_n}^2 = \left( \frac{T}{\pi} \right)^\alpha K(u),
\]
as claimed.
\end{proof}

\begin{proof}[Proof of Proposition \ref{P:focusing_existence}]
For each $\omega\in\RM\setminus\{0\}$, consider the functional
\[
R_\omega(u):=K(u)+\omega Q(u)
\]
and for each $P_0>0$ define the constraint set
\[
\mathcal{B}_{P_0}:=\left\{u\in H^{\alpha/2}_{\rm a}(0,T):P(u)=P_0\right\}.
\]
Observe that by Lemma \ref{L:poincare} we have
\[
R_\omega(u) = K(u) + \omega Q(u) \geq \left( \left( \frac{\pi}{T} \right)^\alpha + \omega \right) Q(u) > 0
\]
and hence that $R_\omega$ is bounded below on $\mathcal{B}_{P_0}$ provided 
$\omega > -\left( \frac{\pi}{T} \right)^\alpha$.  Consequently, for such $\omega$ the number 
$\lambda := \inf_{u \in \mathcal{A}} R_\omega(u)$ is well-defined, so 
there exists a minimizing sequence $\{u_{k}\}_{k=1}^\infty \subset \mathcal{B}_{P_0}$ such that 
$R_\omega(u_{k}) \to \lambda$ as $j \to \infty$. Moreover, using Lemma \ref{L:poincare} again
we find that
\[
R_\omega(u) \geq \left( 1 - \left( \frac{T}{\pi} \right)^\alpha |\omega| \right) K(u)
\]
and hence, using Lemma \ref{L:poincare} again, for $|\omega|<\left(\frac{\pi}{T}\right)^\alpha$ we have
\[
\frac{1}{2} \|u_{k}\|_{H^{\alpha/2}(0,T)}^2 
\leq \left( 1 + \left( \frac{T}{\pi} \right)^\alpha \right) K(u_{k})
\leq \frac{1 + \left( \frac{T}{\pi} \right)^\alpha}{1 - \left( \frac{T}{\pi} \right)^\alpha |\omega|}\;
	R_\omega(u_{k}). 
\]
It follows that for such $\omega$ the sequence $\{u_{k}\}_{k=1}^\infty$ is bounded in $H^{\alpha/2}(0,T)$, 
thus by Banach-Alaoglu there exists a further subsequence $\{u_{k_j}\}_{j=1}^\infty$ 
converging weakly in $H^{\alpha/2}_{\rm a}(0,T)$ and strongly (in norm) in $L^2(0,T)$ 
to some function $\phi \in H^{\alpha/2}_{\rm a}(0,T)$.  
Since $\alpha>1$ the space $H^{\alpha/2}_{\rm a}(0,T)$ is compactly embedded into 
both $L^2(0,T)$ and $L^{2\sigma+2}(0,T)$, and hence that
\[
Q(\phi) = \lim_{j \to \infty} Q(u_{k_j}) \quad \text{and} \quad P(\phi) = \lim_{j \to \infty} P(u_{k_j}) = P_0.
\]
Consequently, $\phi \in \mathcal{B}_{P_0}$.  Finally, since $K$ is weakly lower-semicontinuous on 
$H^{\alpha/2}(0,T)$, we have
\[
\liminf_{j \to \infty} R_\omega(u_{k_j}) \geq K(\phi) + \omega Q(\phi) = R_\omega(\phi)
\]
and thus
\[
\lambda \leq R_\omega(\phi) \leq \liminf_{j \to \infty} R_\omega(u_{k_j}) = \lim_{j \to \infty} R_\omega(u_{k_j}) = \lambda.
\]
It follows that for each $|\omega|<\left(\frac{\pi}{T}\right)^\alpha$ and $P_0>0$, there exists
a nontrivial $\phi\in \mathcal{B}_{P_0}$ such that
\[
R_\omega(\phi) = \inf_{u \in \mathcal{B}_{P_0}} R_\omega(u).
\]
By Lagrange multipliers, there exists $\eta\in \R$ such that
\begin{equation}\label{E:midprofile}
\Lambda^\alpha \phi + \omega \phi + \eta |\phi|^{2\sigma} \phi=0.
\end{equation}

Continuing, observe that multiplying \eqref{E:midprofile} by $\frac{1}{2} \bar{\phi}$ and integrating yields
\begin{align*}
R_\omega(\phi) + (\sigma+1)\eta P(\phi) = 0,
\end{align*}
and hence it must be that $\eta < 0$.  Rescaling $\phi\mapsto|\eta|^{-1/2\sigma}\phi$ we find that $\phi$
solves the focusing fNLS profile equation \eqref{E:focusing_profile_equation} with $c=0$.
Since $R_\omega(a\phi) = a^2 R_\omega(\phi)$, it follows that $\phi$ 
is a constrained minimizer of $R_\omega$ subject to fixed $P \equiv |\eta|^{-(1+1/\sigma)} P_0$.  
Consequently, for every $|\omega|<\left(\frac{\pi}{T}\right)^\alpha$ there exists a function
$\phi$ that minimizes $R_\omega$ subject to a suitably fixed potential energy $P\equiv P_0$.
Note that, as in the defocusing case, 
we may without loss of generality take $\phi$ to be real-valued and even since $P$ and $Q$ 
are preserved under taking absolute value and symmetric decreasing rearrangement, while $K$ 
does not increase under these operations: see Lemma \ref{L:szego}.  Furthermore, following arguments
as in Section \ref{S:existence} we find $\phi\in H^\infty_{\rm a}(0,T)$.
Finally, since 
\begin{align*}
\mathcal{E}(\phi) 
= \inf_{P(u)=P(\phi)} \left( K(u) + \omega Q(u) + P(u) \right) = \inf_{P(u)=P(\phi)} \mathcal{E}(u),
\end{align*}
it follows that the Lagrangian functional $\mathcal{E}(\cdot;\omega)$ is also minimized
by $\phi$ subject to fixed $P\equiv P(\phi)$, completing the proof.
\end{proof}

Observe that the antiperiodic standing waves constructed in Proposition \ref{P:focusing_existence} need
not be constrained minimizers of the Hamiltonian energy $K-P$.  Indeed, using the second derivative test
for constrained extrema in this case implies that
\[
\delta^2 \mathcal{E}(\phi)\Big|_{\{\delta P(\phi)\}^\perp} \geq 0,
\]
where $\delta P(\phi) = |\phi|^{2\sigma}\phi$ is real-valued.  Consequently,
the operator $\delta^2 \mathcal{E}(\phi)$ has at most one negative eigenvalue when
acting on $L^2_{\rm a}(0,T)$.  More precisely, decomposing into real and imaginary parts, the fact 
that $\delta P(\phi)$ is real-valued implies that
\[
n_-(L_+)\leq 1\quad\textrm{and}\quad L_-\geq 0.
\]
Following the spirit of the arguments in Proposition \ref{P:nondegeneracy} above, we can establish
nondegeneracy of both $L_{\pm}$ in this focusing case, as well determine the Morse index of $L_+$.

\begin{proposition}  \label{P:focusing_nondegeneracy}
Let $\alpha \in (1,2)$ and $\sigma > 0$ in the focusing ($\gamma=+1$) fNLS \eqref{E:FNLS1}.  
Let $\phi(\cdot;\omega,P_0) \in H^{\alpha/2}_{\text{a}}(0,T)$ be a real-valued local 
minimizer of $R_\omega$ subject to fixed potential energy $P\equiv P_0$, as constructed in Proposition \ref{P:focusing_existence} and assume, additionally, that $\phi$ depends on $\omega$ in a $C^1$ manner. 
Then the associated Hessian operator
acting on $L^2_{\rm a}(0,T)$ is nondegenerate, i.e.
\[
\ker(\delta^2 \mathcal{E}(\phi)) = \mathrm{span}\{ \phi', i\phi \}
\]
and $n_-(\delta^2 \mathcal{E}(\phi)) = 1$.  Specifically, the operators $L_\pm$ are nondegenerate acting on $L^2_{\text{a}}(0,T)$ with
\[
\ker(L_+) = \mathrm{span}\{\phi'\}, \quad \text{and} \quad \ker(L_-) = \mathrm{span}\{\phi\}
\]
and, further, $n_-(L_+)=1$ and $n_-(L_-)=0$.
\end{proposition}

\begin{proof}
Since $L_-\phi=0$ and $\phi$ is even, it follows from Theorem \ref{T:ground} that $\phi$ is the ground state
of $L_-$ on the even subspace $L^2_{\rm a,even}(0,T)$, and hence $\lambda=0$ is a simple eigenvalue
of $L_-$ restricted to the even subspace.  If $L_-$ were degenerate there would exist
a function $\psi\in L^2_{\rm a,odd}(0,T)$ such that $L_-\psi=0$.  But the ground state theory
Theorem \ref{T:ground} again would imply that $\psi$ may be taken to be strictly positive on $(0,T)$, and hence
cannot be orthogonal to the function $\phi'\in{\rm range}(L_-)$, contradicting the Fredholm alternative.
Consequently, $L_-$ is nondegenerate on $L^2_{\rm a}(0,T)$ and $L_-\geq 0$, as claimed.

Next, we claim that $L_+$ is nondegenerate.  First, note that since $L_+\phi'=0$ and $\phi'$ is strictly
negative on $(0,T)$, it follows by  Theorem \ref{T:ground} that $\phi'$ is the ground
state eigenfunction\footnote{Recall eigenfunctions corresponding to distinct eigenvalues of $L_+$ must be orthogonal, hence sign-definite eigenfunctions correspond to the ground state eigenvalues.} of $L_+$ on the odd subspace $L^2_{\rm a,odd}(0,T)$.  Hence, $\lambda=0$ is a simple
eigenvalue of $L_+$ restricted to the odd subspace.  Since $\phi$ is strictly decreasing on $(0,T)$ and
$\gamma=+1$ here, Proposition \ref{P:gsordering}(ii) implies that the ground state eigenvalue of $L_+$
on $L^2_{\rm a}(0,T)$ is nonpositive and has at least one even eigenfunction. 
To show nondegeneracy, we first show that $L_+$ has exactly one negative eigenvalue.
If $L_+\geq 0$, then by the above discussion $\lambda=0$ must be the ground state eigenvalue
for $L_+$ on the even subspace, and hence there exists a $\psi\in L^2_{\rm a,even}(0,T)$ 
with $\psi(x)>0$ on $(-T/2,T/2)$ such that $L_+\psi=0$.  This, however, contradicts the Fredholm
alternative since such a $\psi$ could not be orthogonal to the function $\phi$, which 
is also strictly positive on $(-T/2,T/2)$ lies in the range of $L_+$ since, differentiating
 \eqref{E:focusing_profile_equation} with $c=0$ with respect to $\omega$ yields
\[
L_+\frac{\partial\phi}{\partial\omega}=-\phi.
\]
Thus, the even ground state eigenvalue of $L_+$ must be strictly
negative, establishing that $n_-(L_+)=1$ as claimed.

To conclude nondegeneracy of $L_+$, note now that $\lambda=0$ must be the second eigenvalue of
$L_+$.  If $L_+$ were degenerate, the simplicity of $\lambda=0$ on the odd subspace $L^2_{\rm a,odd}(0,T)$
implies that there must exist a function $\zeta\in L^2_{\rm a,even}(0,T)$ 
such that $L_+\zeta=0$.  By the oscillation theory Lemma \ref{L:oscillation}, it follows
then that $\zeta$ can change signs at most twice on $(-T/2,T/2)$.  Clearly $\zeta$ cannot 
be sign-definite on $(-T/2,T/2)$ since then $\zeta$ would not be orthogonal to $\phi$, contradicting
again the Fredholm alternative.  Since $\zeta$ is even, it may be normalized so there exists an $x_0\in(0,T/2)$
such that $\zeta(x)>0$ on $(-x_0,x_0)$ and $\zeta(x)<0$ for $x\in(-T/2,-x_0)\cup(x_0,T_0)$.
However, one can easily check that the function
\[
\phi(x)\left( \phi(x)^{2\sigma} - \phi(x_0)^{2\sigma}\right)\in{\rm range}(L_+)
\]
has the same nodal pattern of $\zeta$, again contradicting the Fredholm alternative.  It follows 
that $\ker(L_+|_{L^2_{\text{a,even}}(0,T)}) = \{0\}$, and hence 
$\ker(L_+|_{L^2_{\text{a}}(0,T)}) = \text{span}\{\phi'\}$, as claimed.
\end{proof}

\begin{remark}\label{R:focusing_signdefinite}
Although we do not pursue full arguments here, it is worth noting that the nondegeneracy of the sign
definite solutions of the focusing fNLS with $\omega<0$ follows directly from the analysis
of Hur and Johnson in \cite{HJ15}.  In that case, we would consider positive, $T$-\emph{periodic}
solutions $\phi$ of \eqref{E:focusing_profile_equation} with $c=0$ that would be even, smooth, and strictly
decreasing on $(0,T/2)$.  Such waves could be constructed as critical points of $R_\omega$ above considered
now as acting on $L^2_{\rm per}(0,T)$.  As above,  these 
waves would have $L_-\geq 0$ thanks to the ground state characterization
Theorem \ref{T:ground}, and nondegeneracy of $L_-$ with respect to $T$-periodic perturbations would follow
as above.  For $L_+$ one would have $n_-(L_+)=1$ and nondegeneracy would follow directly from the arguments
in \cite[Section 3]{HJ15}.  The stability of these waves should then follow the same characterization
as described in the forthcoming analysis, depending on the sign of the Jacobian 
$\frac{\partial Q}{\partial\omega}(\phi)$.
\end{remark}

Recapitulating, for the focusing fNLS we have constructed a four-parameter family\footnote{We can parameterize by
$\omega$ and $P_0$, together with the continuous Lie-point symmetries coming from translational and gauge invariances
of the governing evolution equation.} of real-valued, even antiperiodic solutions whose linearizations are necessarily
nondegenerate.  Since such solutions were not constructed as constrained minimizers of the Hamiltonian energy,
however, the nonlinear stability of these waves is not guaranteed as it was for the defocusing analysis.  Rather,
as is common with the local case, the stability depends on the sign of the quantity 
$\frac{\partial Q}{\partial \omega}$.
As a preliminary step in this direction, we show that if $\frac{\partial Q}{\partial\omega}$ 
is positive at the underlying
wave, then this wave is necessarily a constrained minimizer of the Hamiltonian energy subject to fixed
$L^2$-norm.

\begin{lemma}  \label{L:focusing_nonnegative_dQ}
Let $\phi=\phi(\cdot;\omega)$ be a real-valued, even $T$-antiperiodic solution
of the focusing fNLS satisfying the assumptions of Proposition \ref{P:focusing_nondegeneracy}.
If $\frac{\partial Q}{\partial \omega}(\phi) > 0$, then $\phi$ minimizes the Lagrangian functional
$\mathcal{E}$ subject to fixed $L^2$-norm, i.e.
\[
\delta^2 \mathcal{E}(\phi) \big|_{\{ \delta Q(\phi) \}^\perp} \geq 0,
\]
In particular, for such a wave we necessarily have $L_+|_{\{\phi\}^\perp}\geq 0$.
\end{lemma}

\begin{proof}
We use a recent result \cite[Lemma 1]{HSS16} regarding the non-negativity of a self-adjoint operator having a spectral gap and exactly one negative eigenvalue per the above discussion.  We readily verify the hypotheses of their result, as $L_+$ is self-adjoint on the Hilbert space $H_{\text{a}}^{\alpha/2}(0,T)$, and $L_+$ has exactly one negative eigenvalue, which is the (simple) ground state.  Lastly, observe that $\phi \in \ker(L_+)^\perp$, with $L_+^{-1} \phi = -\frac{\partial \phi}{\partial \omega}$.  Then
\[
\innerprod{L_+^{-1} \phi}{\phi} = -\innerprod{\frac{\partial \phi}{\partial \omega}}{\phi} = -\frac{\partial Q}{\partial \omega}(\phi(\cdot; \omega)) < 0,
\]
hence by \cite[Lemma 1]{HSS16} it follows that $L_+|_{\{\phi\}^\perp} \geq 0$.  Since $L_- \geq 0$ a-priori, we conclude that $\delta^2 \CalE(\phi)|_{\{\phi\}^\perp} \geq 0$.
\end{proof}

For the waves constructed in Proposition \ref{P:focusing_existence},we have now established that 
$\delta^2 \mathcal{E}(\phi)$ is non-degenerate, with  
$\ker(\delta^2 \mathcal{E}(\phi)) = \mathrm{span}\{ \phi', i\phi \}$ and that, 
by Lemma \ref{L:focusing_nonnegative_dQ}, that such waves are constrained minimizers of the Hamiltonian
energy subject to fixed $L^2$ norm provided that $\frac{\partial Q}{\partial \omega}(\phi) > 0$.
As in Section \ref{S:stability_of_minimizers}, such conditions are sufficient to conclude the
nonlinear orbital stability of a given wave with respect to $T$-antiperiodic perturbations that slightly change
$Q$.  Indeed, the proof of coercivity is very similar to that of Proposition \ref{P:coercivity}, and the stability result ensues with very few modifications\footnote{In fact, the required analysis follows the stability
theory for the solitary wave case \cite{Wein86}.} to the proof Theorem \ref{T:orbital_stability}.  
Due to the similarly, we only state the result.

\begin{theorem} \label{T:focusing_orbital_stability}
Suppose $\alpha \in (1,2]$, $\sigma>0$ and that $X$ is a suitable subspace of $H^{\alpha/2}_{\text{a}}([0,T]; \C)$ where the Cauchy problem associated with
\[
i u_t - \Lambda^\alpha u + \abs{u}^{2\sigma} u = 0
\]
is well-posed and the functionals $\mathcal{H},Q,N : X \to \R$ are smooth.
Further, let $\phi_0 \in H^{\alpha/2}_{\rm a}(0,T)$ be a real-valued $T$-antiperiodic standing solution of \eqref{E:focusing_profile_equation} satisfying the hypotheses of Proposition \ref{P:focusing_nondegeneracy}.
If 
\[
\frac{\partial Q}{\partial \omega}(\phi_0) > 0,
\]
then for all $\varepsilon > 0$ sufficiently small there exists a constant $C=C(\varepsilon)$ such that if $v \in X$ with $\|v\|_X \leq \varepsilon$ and $u(\cdot,t)$ is a local in time solution of 
\[
i u_t -\omega u- \Lambda^\alpha u + \abs{u}^{2\sigma} u = 0
\]
with initial data $u(\cdot,0) = \phi_0 + v$, then $u(\cdot,t)$ can be continued to a solution for all $t > 0$ and
\[
\sup_{t > 0} \inf_{(\beta,x_0) \in \R^2} \left\| u(\cdot,t) - e^{i\beta} \phi_0(\cdot - x_0) \right\|_X \leq C\|v\|_X.
\]
\end{theorem}

We remark that the criterion $\frac{\partial Q}{\partial \omega}(\phi_0) > 0$ in Theorem 
\ref{T:focusing_orbital_stability} is a sufficient condition for 
orbital stability in the focusing case, compensating for the fact that $\phi$ was not constructed as a minimizer of $\mathcal{E}$ with fixed $Q$: see Lemma \ref{L:focusing_nonnegative_dQ}.  
Conversely, $\phi_0$ can be shown to be spectrally unstable if $\frac{\partial Q}{\partial \omega}(\phi_0) < 0$.

\begin{proposition}\label{P:focusing_instability}
Let $\phi_0$ be as in Theorem \ref{T:focusing_orbital_stability}.  If $\frac{\partial Q}{\partial \omega}(\phi_0) < 0$, then $\phi_0$ is spectrally unstable.
\end{proposition}

The proof follows the now standard Vakhitov-Kolokolov projection method, and the interested reader is referred
to \cite[Section 4.1]{SS99} for details.

\begin{appendix}


\section{Antiperiodic Rearrangement Inequalities}\label{A:Polya}

In this section, we establish results pertaining to symmetric decreasing rearrangements
of $T$-antiperiodic functions and their consequences.
Given a function $f\in L^2_{\rm per}([0,2T];\RM)\cup C^0(\RM)$ we will end up utilizing \emph{four separate 
equimeasurable rearrangements} of $f$, which we will describe below.

\

Throughout, we denote by $m$ the Lebesgue measure on $\RM/2T\ZM$.  Given a continuous $2T$-periodic
function $f:\RM\to\RM$, we define the $2T$-periodic symmetric decreasing rearrangement $f^{*_{2T}}$ of $f$
on $(-T,T)$ by
\[
f^{*_{2T}}(x)=\inf\left\{t:m(\{z\in(-T,T):f(z)>t\})\leq 2|x|\right\}\quad\textrm{for}~~x\in[-T,T]
\]
and note that $f^{*_{2T}}$ is even, nonincreasing on $(0,T)$, and satisfies $f^{*_{2T}}(0)=\max_{x\in\RM}f(x)$
and $f^{*_{2T}}(T)=\min_{x\in\RM}f(x)$.  Similarly, we define the $2T$-periodic rearrangement $f^{\#_{2T}}(x)$
by
\[
f^{\#_{2T}}(x)=f^{*_{2T}}(x-T/2)
\]
and note that $f^\#$ is even about $x=T/2$, nondecreasing on $(-T/2,T/2)$ and satisfies
$f^{\#_{2T}}(T/2)=\max_{x\in\RM}f(x)$ and $f^{\#_{2T}}(-T/2)=\min_{x\in\RM}f(x)$.  Both $f^{*_{2T}}$ 
and $f^{\#_{2T}}$ have the same
distribution functions as $f$ on $(-T,T)$ so that, in particular,
\[
\|f\|_{L^p(-T,T)}=\|f^{*_{2T}}\|_{L^p(-T,T)}=\|f^{\#_{2T}}\|_{L^p(-T,T)}
\]
for all $f\in L^p_{\rm per}(0,T)$ and $1\leq p\leq\infty$.  Of special interest here is that if $f$ is $T$-antiperiodic,
then $f^{*_{2T}}$ is an even, $T$-antiperiodic function on $\RM$ while $f^{\#_{2T}}$ is an odd, $T$-antiperiodic
function on $\RM$.

Our first result is an analogue of the classical P\'olya-Szeg\"o inequality, which states that the kinetic
energy is nonincreasing under symmetric decreasing rearrangement.

\begin{lemma}[P\'olya-Szeg\"o]\label{L:szego}
For all $\alpha\in(1,2)$ and $f\in H^{\alpha/2}_{\rm per}([0,2T];\mathbb{R})$, we have
\[
\int_{-T}^T\left|\Lambda^{\alpha/2} f^{*_{2T}}\right|^2dx=
\int_{-T}^T\left|\Lambda^{\alpha/2} f^{\#_{2T}}\right|^2dx \leq\int_{-T}^T\left|\Lambda^{\alpha/2} f\right|^2dx.
\]
In particular, if such an $f$ is $T$-antiperiodic, then
\[
\int_{-T/2}^{T/2}\left|\Lambda^{\alpha/2} f^{*_{2T}}\right|^2dx=
\int_{-T/2}^{T/2}\left|\Lambda^{\alpha/2} f^{\#_{2T}}\right|^2 
	\leq\int_{-T/2}^{T/2}\left|\Lambda^{\alpha/2} f\right|^2dx.
\]
\end{lemma}

\begin{proof}
Given $f\in H^{\alpha/2}_{\rm a}([0,T];\RM)$, observe that for all $t>0$ we have
\[
\left<f,e^{-\Lambda^\alpha t}f\right> = \int_{-T}^T\int_{-T}^T f(x)K_p(x-y,t)f(y)dx~dy,
\]
where here $K_p(x,t)$ is the $2T$-periodic integral kernel associated to the semigroup $e^{-\Lambda^\alpha t}$
defined in \eqref{E:Kp_rep1}.  By Lemma \ref{L:properties_of_Kp}, $K_p(\cdot,t)=K_p^{*_{2T}}(\cdot,t)$
for all $t>0$ and hence the Bernstein-Taylor Theorem \cite[Theorem 2]{BT76} we have
\[
\int_{-T}^T\int_{-T}^T f(x)K_p(x-y,t)f(y)dx~dy\leq \int_{-T}^T\int_{-T}^T f^{*_{2T}}(x)K_p(x-y,t)f^{*_{2T}}(y)dx~dy
\]
so that
\[
\left<f,e^{-\Lambda^\alpha t}f\right>\leq \left<f^{*_{2T}},e^{-\Lambda^\alpha t}f^{*_{2T}}\right>
\]
for all $f\in H^{\alpha/2}_{\rm a}([0,T];\RM)$ and $t>0$.  Since $f$ and $f^{*_{2T}}$ 
are equimeasurable, it follows that
\[
\frac{\left<f,e^{-\Lambda^\alpha t}f\right>-\|f\|_{L^2(-T,T)}^2}{t}\leq
\frac{\left<f^{*_{2T}},e^{-\Lambda^\alpha t}f^{*_{2T}}\right>-\|f^{*_{2T}}\|_{L^2(-T,T)}^2}{t}
\]
for all $t>0$.  Taking $t\to 0^+$ yields the desired result for the rearrangement $f^{*_{2T}}$.  
The corresponding result for $f^{\#_{2T}}$ and the restriction to $T$-antiperiodic functions
now follows trivially.
\end{proof}

Next, we complement the above result by considering the effect of the above rearrangements
on linear potentials.

\begin{lemma}\label{L:potentialrearrange}
Let $V:\RM\to\RM$ be an even, smooth and $T$-periodic potential.  
\begin{itemize}
\item[(i)] If $V(x)$ is nonincreasing on $(0,T/2)$, then
\[
\int_{-T/2}^{T/2}V(x) f^2(x)dx\geq\int_{-T/2}^{T/2}V(x)\left(f^{\#_{2T}}\right)^2(x)dx
\]
for all continuous $f\in L^2_{\rm a}([0,T];\RM)$.
\item[(ii)] If $V(x)$ is nondecreasing on $(0,T/2)$, then
\[
\int_{-T/2}^{T/2}V(x) f^2(x)dx\geq\int_{-T/2}^{T/2}V(x)\left(f^{*_{2T}}\right)^2(x)dx
\]
for all continuous $f\in L^2_{\rm a}([0,T];\RM)$.
\end{itemize}
\end{lemma}

\begin{proof}
We begin by proving (i).  Notice by the hypothesis on $V$, the function $(-V(x))$ is even about $x=T/2$
and is nonincreasing on $(T/2,T)$.  By the Riesz inequality \cite[Section 3.4]{LL} we thus have
\[
\int_0^T (-V(x))f^2(x)dx\leq\int_{0}^T (-V(x))\left(f^2\right)^{\#_T}(x)dx,
\]
where here $(f^2)^{\#_T}$ denotes the $T$-periodic rearrangement of the $T$-periodic
function $f^2$ taken to be even about $x=T/2$ and nonincreasing on $(T/2,T)$.  Since
antiperiodicity of $f$ implies
\[
\left(f^2\right)^{\#_T}(x)=\left(f^{\#_{2T}}\right)^2(x)\quad\forall x\in(0,T),
\]
the estimate in (i) follows.

Similarly, if $V$ satisfies the hypotheses of (ii), then the Riesz inequality again gives
\[
\int_0^T \left(-V(x)\right)f^2(x)dx\leq\int_{0}^T \left(-V(x)\right)\left(f^2\right)^{*_T}(x)dx,
\]
where here $(f^2)^{*_T}$ denotes the $T$-periodic rearrangement of the $T$-periodic 
function $f^2$ taken to be even about $x=0$ and nonincreasing on $(0,T/2)$.  Antiperiodicity
of $f$ again implies that
\[
\left(f^2\right)^{*_T}(x)=\left(f^{*_{2T}}\right)^2(x)\quad\forall x\in(0,T/2),
\]
the estimate in (ii) follows.
\end{proof}

We now come to the main result of this appendix, providing an ordering between the even and odd ground
state antiperiodic eigenvalues of a periodic Schr\"odinger operator $L=-\Lambda^\alpha+V$ in terms of the monotonicity
properties of the potential $V$.

\

\noindent\textit{Proof of Proposition \ref{P:gsordering}}:
First, assume that $V(x)$ satisfies the hypothesis of (i) and suppose that $\psi$ is an eigenfunction
associated to the ground state eigenvalue of $L$ acting on $L^2_{\rm a,even}(0,T)$, normalized
to be real-valued and $\|\psi\|_{L^2(0,T)}=1$.  Then by Lemma \ref{L:szego} and Lemma \ref{L:potentialrearrange}
\begin{align*}
\min\sigma\left(L\big{|}_{L^2_{\rm a,even}(0,T)}\right)&=
\int_0^T\left|\Lambda^{\alpha/2}\psi\right|^2dx+\int_0^T V(x)\psi(x)^2dx\\
&\geq\int_0^T\left|\Lambda^{\alpha/2}\psi^{\#_{2T}}\right|^2dx+\int_0^TV(x)\left(\psi^{\#_{2T}}\right)^2(x)dx\\
&\geq\min\sigma\left(L\big{|}_{L^2_{\rm a,odd}(0,T)}\right),
\end{align*}
where the last inequality is justified since $\|\psi^{\#_{2T}}\|_{L^2(0,T)}=1$.  This verifies (i).
A similar proof establishes the ordering in (ii).

\end{appendix}

\section*{Acknowledgments}  \label{S:acknowledgments}
The authors would like to thank Thierry Gallay for useful discussions regarding the work
\cite{GH07}, as well as Mats Ehrnstr\"om and Erik Whalen for guiding us to the reference \cite{Andersson13}.  
We also thank Dmitry Pelinovsky, Atanas Stefanov, and Younghun Hong for several valuable discussions.
Finally, the authors are grateful to the anonymous  referee for their careful reading of the manuscript
and for several helpful and insightful suggestions.
The research of MAJ was supported by the National Science Foundation under
grants DMS-1614785 and DMS-1211183.  The research of KMC was supported by the National Science Foundation under
grant DMS-1211183.


\section*{Conflict of Interest}
This study was funded by the National Science Foundation under grant DMS-1211183 and DMS-1614785.

\bibliographystyle{plain}
\bibliography{FNLS_stability_antiper}

\end{document}